\documentclass[11pt]{article}
\usepackage{amsmath,amssymb,a4wide,amsthm,color,graphicx}
\usepackage{verbatim}

\newtheorem{theorem}{Theorem}[section]
\newtheorem{corollary}[theorem]{Corollary}
\newtheorem{remark}[theorem]{Remark}

\newtheorem{proposition}[theorem]{Proposition}
\newtheorem{ex}[theorem]{Example}
\newenvironment{example}{\begin{ex} \em }{\em \end{ex}}
\newtheorem{definition}[theorem]{Definition}

\newcommand{\norm}  [1]{\ensuremath{\left  \|       #1  \right \|       }}

\newcommand{\sqb}   [1]{\ensuremath{\left [        #1  \right ]        }}

\newcommand{\cb}    [1]{\ensuremath{\left  \{      #1  \right \}       }}
\newcommand{\cbg}   [1]{\ensuremath{\bigl  \{      #1  \bigr  \}       }}
\newcommand{\cbgg}  [1]{\ensuremath{\biggl \{      #1  \biggr \}       }}

\newcommand{\of}    [1]{\ensuremath{\left (        #1  \right )        }}
\newcommand{\ofg}   [1]{\ensuremath{\bigl (        #1  \bigr  )        }}

\newcommand{\abs  } [1]{\ensuremath{\left  |       #1  \right |        }}

\newcommand{\st} {\ensuremath{|\;}}

\newcommand{\dual} {{\rm dual \,}}

\newcommand{\cl}  {{\rm cl  \,}}

\newcommand{\bd}  {{\rm bd \,}}

\newcommand{\Int} {{\rm int \,}}
\newcommand{\conv}  {{\rm conv \,}}

\newcommand{\cone}{{\rm cone\,}}
\newcommand{\gr}  {{\rm gr\,}}

\DeclareMathOperator*{\Max}{ Max}
\DeclareMathOperator*{\Min}{ Min}

\DeclareMathOperator*{\wMin}{wMin}

\newcommand{\T}{\mathcal{T}}

\newcommand{\D}{\mathcal{D}}

\renewcommand{\P}{\mathcal{P}}

\newcommand{\R}{\mathbb{R}}
\newcommand{\N}{\mathbb{N}}

\renewcommand{\subset}{\subseteq}
\renewcommand{\supset}{\supseteq}

\newcommand{\smz}{\setminus\{0\}}

\newcommand{\One}{\mathrm{1\negthickspace I}}

\newcommand{\diag}{{\rm diag}}

\newcommand{\bs}{\backslash}

\author{Andreas H. Hamel \and Andreas L\"{o}hne \and Birgit Rudloff}

\title{Benson type algorithms for linear vector optimization and applications}

\date{February 11, 2013 (update: July 28, 2013)}

\begin{document}
\maketitle

\begin{abstract} 
New versions and extensions of Benson's outer approximation algorithm for solving linear vector optimization problems are presented. Primal and dual variants are provided in which only one scalar linear program has to be solved in each iteration rather than two or three as in previous versions. Extensions are given to problems with arbitrary pointed solid polyhedral ordering cones. Numerical examples are provided, one of them involving a new set-valued risk measure for multivariate positions.

\medskip


\medskip


\end{abstract}

\section{Introduction}
Set-valued approaches to vector optimization are promising in theory and applications. A duality theory in this framework is important for algorithms, and the dual problems can be interpreted in certain applications, see e.g. \cite{EhrLoeSha12,Hamel09,HamHey10,HamHeyLoeTamWin04,HamHeyRud11,HeyLoe08,HeyLoeTam09,HeyLoeTam09-1, Loehne11book,ShaEhr08-1}. Benson's outer approximation algorithm is a fundamental tool for solving linear (and also convex) vector optimization problems \cite{Benson98a,Benson98,EhrLoeSha12,ShaEhr08,ShaEhr08-1}. It is also important for solving set-valued problems \cite{LoeSch12}. Recent applications of linear vector optimization concern financial markets with frictions (transaction costs). For such applications, one obtains optimization problems which are genuinely set-valued. The need to compute the values of a set-valued risk measure for multi-variate random variables was a driving force for this work. 

In this article, we introduce a primal and a dual algorithm of Benson type where only one LP has to be solved in each iteration step\footnote{A similar variant has been developed independently in \cite{Csirmaz13}.}. In contrast, previous versions \cite{Benson98a,Benson98,EhrLoeSha12,ShaEhr08,ShaEhr08-1} require at least two different LPs in each step. As the main effort of Benson type algorithms in typical applications is caused by the LPs, the computational time can be reduced considerably by the new algorithms.
Another advantage is that all LPs have a very similar structure and therefore the impact of warm starts can be improved. A further benefit is an improvement of the error estimation given in \cite{ShaEhr08,ShaEhr08-1}, i.e., in approximate variants of the algorithms: The same approximation error can be achieved with fewer iteration steps (compare Remark \ref{rem_eps} and Example \ref{ex07} below). For both the primal and dual algorithm two variants (`break' and `no break') are presented and compared (compare Example \ref{ex07}). Another novelty of this article is that linear vector optimization problems with arbitrary polyhedral solid pointed ordering cones are treated, whereas in all other references \cite{Benson98a,Benson98,EhrLoeSha12,EhrShaSch11,Loehne11book,ShaEhr08,ShaEhr08-1} only the special case of the usual ordering cone $\R^q_+$ is considered. This feature will be exploited in applications involving set-valued risk measures for multi-variate random variables in markets with transaction costs. In such situations, ordering cones are usually different from $\R^q_+$ and generated by a large number of directions. A short introduction into this topic and several (numerical) examples are given. Examples \ref{ex_WCtheory} and \ref{ex45} involve a new type of a set-valued risk measure which we baptized the `relaxed' worst case risk measure.

This article is organized as follows. In Section 2 we provide some basic notations and results. The next three sections start with short introductions. Section 3 contains an overview on the set-valued approach to linear vector optimization and related duality results where, in contrast to most of the literature, we allow ordering cones more general than $\R^q_+$. In Sections 4 we introduce the new variants of Benson's algorithm. We also give a detailed description of the two-phase-method to treat unbounded problems. Section 5 provides an introduction to applications involving set-valued risk measures, and in Section 6 several numerical examples are reported.

\section{Preliminaries}

Let $A\subset \R^q$. We denote by $\cl A$, $\Int A$, $\bd A$ the closure, interior and boundary of $A$, respectively. The set $A$ is called {\em solid} if its interior is non-empty. A {\em convex polyhedron} or {a polyhedral convex set} $A$ in $\R^q$ is defined to be the intersection of finitely many half spaces, that is
\[
A=\bigcap_{i=1}^r \cb{y \in \R^q \st (z^i)^T y \geq \gamma_i} 
\]
where $z^1, \ldots, z^r \in \R^q\bs\cb{0}$ and $\gamma_1, \ldots, \gamma_r \in \R$.
As polyhedra considered in this article are always convex, we will not mention convexity explicitly. Every non-empty polyhedron $A \subseteq \R^q$ can be expressed by means of a finite number of points $y^1, \ldots, y^s \in \R^q$ ($s\in \N\smz$) and directions $k^1, \ldots, k^t \in \R^q\setminus \cb{0}$ ($t \in \N$) through
\[ 
A = \cbgg{ \sum_{i=1}^s \lambda_i y^i + \sum_{j=1}^t \mu_j k^j \bigg|\; \lambda_i \geq 0,\; i=1,\dots,s,\; \sum_{i=1}^s \lambda_i = 1,\; \mu_j \geq 0,\;j=1,\dots,t},
\]
where $k \in \R^q\smz$ is called a direction of $A$ if $A + \cb{\mu\cdot k} \subseteq A$ for all $\mu>0$.
This can be also written as
\begin{equation}\label{eq1}
 A = \conv \cb{y^1, \ldots,y^s} + \cone\cb{k^1, \ldots, k^t}.
\end{equation}
Note that we set $\cone \emptyset = \cb{0}$. The polyhedron $A$ is bounded if and only if the cone-part in the above formula is $\cb{0}$. The vectors \cb{y^1, \ldots, y^s} and directions \cb{k^1, \ldots, k^t} are called the {\em generators} of $A$. The set $A_\infty := \cone\cb{k^1, \ldots, k^t}$ is the recession cone of $A$. A finite set of half spaces defining a polyhedron $A$ is called H-representation (or inequality representation) of $A$, whereas a finite set of points and directions defining $A$ is called V-representation (or generator representation) of $A$. A bounded polyhedron is called a {\em polytope}. A convex subset $F$ of a convex set $A$ is called a {\em face} of $A$ if
$\of{\bar y,\hat y\in A \,\wedge\, \lambda\in (0,1)\,\wedge\,\lambda \bar y + (1-\lambda) \hat y\in F}$ implies $y,\hat y\in F$. A set $F$ is a {\em proper} (i.e. $\emptyset\neq F\neq A$) face of a polyhedron A if and only if there is a supporting hyperplane $H$ to $A$ with $F=H\cap A$. The proper $(r-1)$-dimensional faces of an $r$-dimensional polyhedral set $A$ are called {\em facets} of $A$. A point $y\in A$ is called a {\em vertex} of $A$ if $\cb{y}$ is a face of $A$. If $k \in \R^q\setminus \cb{0}$ belongs to a half-line face of a polyhedral set $A$, then $k$ is called {\em extreme direction} of $A$.

A polyhedral convex cone $C \subset \R^q$ is called {\em pointed} if it contains no lines. Of course, a solid and pointed convex cone is {\em non-trivial}, that is, $\cb{0} \subsetneq C \subsetneq \R^q$. 
A non-trivial convex pointed cone $C \subset\R^q$ defines a partial ordering $\leq_C$ on $\R^q$ by $y^1 \leq_C y^2$ if and only if $y^2-y^1 \in C$. If $C=\R^q_+ :=\cb{y\in \R^q \st y_1 \geq 0, \ldots , y_q \geq 0}$, then the component-wise ordering $\leq_{\R^q_+}$ is abbreviated to $\leq$. A point $y \in \R^q$ is called {\em $C$-minimal in $A$} if $y \in A$ and $(\cb{y} -C\smz) \cap A = \emptyset$. The set of $C$-minimal elements of $A$ is denoted by $\Min_C A$. If $C$ is additionally solid (but not necessarily pointed), a point $y \in \R^q$ is called {\em weakly $C$-minimal in $A$} if $y \in A$ and $(\cb{y} - \Int C) \cap A = \emptyset$. Likewise, by replacing $C$ by $-C$, {\em $C$-maximal} and {\em weakly $C$-maximal} points in a set $A \subseteq \R^q$ are introduced and we write $\Max_C A$ for the set of $C$-maximal elements of $A$. The dual cone of a cone $C \subseteq \R^q$ is the set $C^+:=\cb{y^* \in \R^q \st \forall y \in C: (y^*)^T y \geq 0}$. The $i$-th canonical unit vector in $\R^q$ is denoted by $e^i$.

\section{Linear vector optimization}

In this section we outline the set-valued approach to linear vector optimization including duality theory and establish a more general setting where arbitrary polyhedral ordering cones $C\subseteq\R^q$ rather than $C=\R^q_+$ are supposed. A comprehensive 
exposition for the case of the ordering cone $C=\R^q_+$ can be found in \cite{Loehne11book}. The origin of this approach is discussed in \cite[Section 4.8]{Loehne11book}. A related duality theory and an overview on other approaches to duality for linear vector optimization problems can be found in a recent paper by Luc \cite{Luc11}.

\subsection{Problem setting and solution concepts}

The solution concepts defined in this section are based on the idea that in vector optimization (in contrast to scalar optimization), minimality and infimum attainment are no longer equivalent concepts. In order to have an appropriate complete lattice where an infimum is defined and exists, a set-valued reformulation of the vector optimization problem is necessary. Here we just introduce the solution concepts that result from these ideas. We motivate these concepts from an application oriented viewpoint only. More details and a theoretical motivation can be found in \cite{HamLoe12,HeyLoe11, Loehne11book}. 

Let matrices $B \in \R^{m\times n}$, $P \in \R^{q \times n}$, a vector $b \in \R^m$ and a  solid pointed polyhedral cone $C \subseteq \R^q$ be given. The following linear vector optimization problem is considered:

\begin{equation}\tag{P}\label{p}
 \text{ minimize } P: \R^n\to\R^q \text{ with respect to } \le_C \text{ subject to } Bx \geq b.
\end{equation}
Define
\[
S = \cb{x \in \R^n \mid Bx \geq b} \quad \text{and} \quad 
	S^h = \cb{x \in \R^n \mid Bx \geq 0}.
\]
Of course, we have $S^h = S_\infty$, that is, the non-zero points in $S^h$ are exactly the directions of $S$.
A point $\bar x \in S$ is said to be a {\em minimizer} for \eqref{p} if there is no $x \in S$ with $P x \leq_C P \bar x$ and $P x \neq P\bar x$, that is, $P \bar x$ is $C$-minimal in $P[S]:=\cb{Px\st x \in S}$. The set of minimizers of \eqref{p} is denoted by $\Min\eqref{p}$.

A direction $k \in \R^n\smz$ of $S$ is called a {\em minimizer} for \eqref{p} if the corresponding point $k \in S^h\smz$ is a minimizer of the homogeneous problem
\begin{equation}\tag{P$^h$}\label{ph}
 \text{ minimize } P: \R^n\to\R^q \text{ with respect to } \le_C \text{ subject to } Bx \geq 0.
\end{equation}

\begin{definition}
A set $\bar S \subseteq S$ is called a set of {\em feasible} points for \eqref{p} and, whenever $S \neq \emptyset$, a set $\bar S^h \subseteq S^h\smz$ is called a set of feasible directions for \eqref{p}. 

A pair of sets $\of{\bar S, \bar S^h}$ is called a {\em finite infimizer} for \eqref{p} if $\bar S$ is a non-empty finite set of feasible points for \eqref{p}, $\bar S^h$ is a (not necessarily non-empty) finite set of feasible directions for \eqref{p}, and
\begin{equation}\label{eq_inf_att}
\conv P[\bar S]  + \cone P[\bar S^h] +C = P[S] + C.
\end{equation}
\end{definition}

An infimizer can be understood as a feasible element for a set-valued extension of problem \eqref{p} (lattice extension) where the infimum (which is well defined for this lattice extension) is attained, i.e., condition \eqref{eq_inf_att} stands for the {\em infimum attainment}. 

The set $\P:=P[S]+C$ is called {\em upper image} of \eqref{p}. Clearly, if $(\bar S, \bar S^h)$ is a finite infimizer and $(\cb{0},Y)$ is a V-representation of the cone $C$, then $(P[\bar S],P[\bar S^h] \cup Y)$ is a V-representation of the upper image $\P$.

The following solution concept is based on a combination of minimality and infimum attainment.

\begin{definition}\label{def_sol}
A finite infimizer $(\bar S, \bar S^h)$ of \eqref{p} is called a {\em solution} to \eqref{p} if $\bar S$ and $\bar S^h$ consist of only minimizers.
\end{definition}

In practice, the upper image $\P$ is one of the most important information for a decision maker. This is due to the fact that in typical applications the dimension $n$ of the decision space is considerably larger than the dimension $q$ of the outcome (or criteria) space. The problem to calculate all the minimizers is usually not tractable. Moreover, the overwhelming set of all minimizers is in general not suitable to support a decision. It is more natural and easier in practice to compare the criteria rather than decisions. Furthermore, also in scalar programming it is often not necessary to know all optimal solution. A solution as introduced above can be seen as an outcome set based concept which provides the information to describe the upper image $\P$.

\subsection{Duality}

If vector optimization is considered in a set-valued framework, it is very natural to consider a dual problem with a hyperplane-valued objective function. First, we note that this is also the case in scalar optimization as in $\R$ a point and a hyperplane are the same thing. Secondly, we have in mind the well-known dual description of a convex set by hyperplanes. The upper image $\P$ of a linear vector optimization problem \eqref{p} is a convex polyhedron which can be interpreted as an infimum of the lattice extension of problem \eqref{p}, see \cite{Loehne11book}. It is therefore natural to ask for a dual description of this convex set which is obtained as the supremum of a suitable dual problem. As a third argument, we want to mention that there is a lack of applications of the classical approaches to duality theory in vector optimization. For instance, Ehrgott \cite{Ehrgott_NC} pointed out that ``dual algorithms could not be developed because of the absence of a duality theory for MOLP that could be algorithmically exploited.''

The idea of {\em geometric duality} \cite{HeyLoe08} is to transform the hyperplane-valued dual problem into a vector optimization problem. This idea is taken from the theory of convex polytopes, where an H-representation of a polytope $A$ defines a V-representation of a dual polytope. For instance, if $A$ is solid and contains zero in its interior, an H-representation of the form
\[ A=\cb{x\in \R^q \st B x \leq (1,\dots,1)^T } \]
exists. The row vectors of the matrix $B$ yield a V-representation of the polar set $A^\circ:=\{y^*\in\R^q \st \forall y \in A: {y^*}^T y \leq 1 \}$ of $A$, which is a dual polytope to the polytope $A$. The duality
relation between $A$ and the dual polytope $A^\circ$ is given by an inclusion reversing one-to-one map $\Psi$ between the set of all faces of $A$ and the set of all faces of $A^\circ$.

A similar duality map can be used to transform a hyperplane-valued optimization problem into a vector-valued problem, which is called the {\em geometric dual problem}. We assume throughout that there exists a vector
\begin{equation}\label{ass_c}
 c \in \Int C \quad \text{ such that } \quad c_q=1,
\end{equation}
and we fix such a vector $c$. As $C$ was assumed to be a solid cone, there always exists some $c \in \Int C$ such that either $c_q=1$ or $c_q =-1$. In the latter case we can consider the problem where $C$ and $P$ are replaced by $-C$ and $-P$, which is equivalent to (P) and fulfills \eqref{ass_c}.

Consider the dual problem
\begin{equation*}\label{d}
\tag{D$^*$} \text{ maximize } D^*: \R^m \times \R^q \to \R^q \text{ with respect to } \leq_K \text{ over } T,
\end{equation*}
with (linear) objective function
\[ D^*:\R^m\times\R^q \to \R^q,\quad D^*(u,w):=\of{w_1,...,w_{q-1},b^T u}^T,\]
ordering cone
\[ K:=\R_+ \cdot (0,0,\dots,0,1)^T,\]
and feasible set
\[ T:=\cb{(u,w)\in \R^m \times \R^q \st u \geq 0,\; B^T u = P^T w,\; c^T w = 1,\; Y^T w \geq 0},\]
where $Y$ is the matrix whose columns are generators of the ordering cone $C$.

A point $(\bar u, \bar w) \in T$ is said to be a {\em maximizer} for \eqref{d} if there is no $(u,w) \in T$ with $D^*(u,w) \geq_K D^*(\bar u,\bar w)$ and $D^*(u,w) \neq D^*(\bar u,\bar w)$. The set of maximizers of \eqref{d} is denoted by $\Max\eqref{d}$.

\begin{definition}[\cite{Hamel11,HamLoe12,HeyLoe11,Loehne11book}]
A non-empty finite set $\bar T$ of points being feasible for \eqref{d} is called a {\em finite supremizer} of \eqref{d} if
\begin{equation}\label{eq_sup_att}
 \conv D^*[\bar T] - K = D^*[T] -K.
\end{equation}
\end{definition}

The set $\D^*= D^*[T]-K$ is called {\em lower image} of \eqref{d}. Similarly to above, a finite supremizer yields a V-representation of $\D^*$. Condition \eqref{eq_sup_att} can be interpreted as the attainment of the supremum in a suitable complete lattice, see e.g. \cite{Loehne11book}.
The combination of maximality and supremum attainment leads to a solution.

\begin{definition}\label{def_sold}
A finite supremizer $\bar T$ of \eqref{d} is called a {\em solution} to \eqref{d} if it consists of only maximizers.
\end{definition}
Note that, in contrast to \eqref{eq_inf_att}, we do not need directions in \eqref{eq_sup_att}. This is due to the simplicity of the cone $K$ in contrast to $C$. 

A duality mapping $\Psi$ for the vector optimization problems \eqref{p} and \eqref{d} is now introduced. The bi-affine function
\[ \varphi:\; \R^q\times\R^q \to \R,\quad \varphi(y,y^*):=\sum_{i=1}^{q-1} y_i y^*_i + y_q \of{1-\sum_{i=1}^{q-1} c_i y^*_i} - y^*_q\]
is used to define the following two injective hyperplane-valued maps
\[ H:  \R^q \rightrightarrows \R^q,\quad  H(y^*):= \cb{y \in \R^q \st   \varphi(y,y^*)=0},\]
\[ H^*:\R^q \rightrightarrows \R^q,\quad  H^*(y):= \cb{y^* \in \R^q \st \varphi(y,y^*)=0}.\]
The mapping $H$ yields the duality map
\[\Psi: 2^{\R^q} \to 2^{\R^q}, \quad\Psi( F^*):=\bigcap_{y^* \in  F^*}  H(y^*) \cap \P.\]
By setting
\begin{equation}\label{eq_w}
w(y^*):= \of{y^*_1,\dots,y^*_{q-1},1-\sum_{i=1}^{q-1} c_i y^*_i}
\end{equation}
and
\[ w^*(y):=\of{y_1-y_q c_1\,,\dots,\,y_{q-1}-y_q c_{q-1}\,,\,-1} ,\]
we can write
\begin{equation}\label{eq_phiw}
\varphi(y,y^*)= w(y^*)^T y -y^*_q = w^*(y)^T y^* + y_q,
\end{equation}
which is useful for the geometric interpretation of duality.

Weak duality reads as follows.
\begin{theorem}[\cite{HeyLoe08,Loehne11book}]\label{th_wgd}
The following implication is true:
\[ \sqb{y \in \P \wedge y^* \in \D^*} \implies \varphi(y,y^*)\geq 0.\]
\end{theorem}
Note that weak duality implies the inclusions
\[\D^* \subseteq \cb{y^* \in \R^q \st \forall y \in \P: \; \varphi(y,y^*)\geq 0} \quad
\text{and}
\quad \P \subseteq \cb{y \in \R^q \st \forall y^* \in \D^*: \; \varphi(y,y^*)\geq 0}, \]
whereas the following strong duality theorem yields even equality.

\begin{theorem}[\cite{HeyLoe08,Loehne11book}]\label{th_fsd}
Let the feasible sets $S$ of \eqref{p} and $T$ of \eqref{d} be non-empty. Then
\begin{align*}
	\sqb{\forall y^* \in \D^*: \varphi(y,y^*)\geq 0} &\implies y   \in \P    ,\\
    \sqb{\forall y \in \P:     \varphi(y,y^*)\geq 0} &\implies y^* \in \D^*.
\end{align*}
\end{theorem}

\begin{remark}\label{rem_3.7}
Theorems \ref{th_wgd} and \ref{th_fsd} are formulated and proved in the mentioned references only for the special case $C=\R^q_+$ and $c=(1,\dots,1)^T$. However, using the generalized variants of scalarized problems \eqref{p1}, \eqref{d1}, \eqref{p2} and \eqref{d2} as stated below, the general case can be obtained in the same way. 
\end{remark}

The following {\em geometric duality theorem} provides a third type of duality relation. It takes into account the facial structure of the sets $\P$ and $\D^*$. Note that geometric duality does not play any role in scalar optimization because the structure of polyhedra in the objective space $\R$ is very simple. 

\begin{theorem}[\cite{Heyde11,HeyLoe08}]\label{th_mr}
$\Psi$ is an inclusion reversing one-to-one map between the set of all $K$-maximal proper faces of $\D^*$ and the set of all proper faces of $\P$. The inverse map is given by
\begin{equation*}
 \Psi^{-1}(F)=\bigcap_{y \in  F}  H^*(y) \cap \D^*.
\end{equation*}
Moreover, if $F^*$ is a $K$-maximal proper face of $\D^*$, then
\[ \dim  F^* + \dim \Psi( F^*) = q-1.\]
\end{theorem}

\begin{remark}\label{rem_3.9}
The proof of the special case $c=(1,...,1)^T$ and $C=\R^q_+$ can be found in \cite{HeyLoe08}. The general case can be proved in the same way using the generalized versions of \eqref{p1}, \eqref{d1}, \eqref{p2} and \eqref{d2} as defined below. Theorem \ref{th_mr} (in the general setting) is also a special case of geometric duality theorem for convex vector optimization problems, see Example 3 in \cite{Heyde11}.
\end{remark}

\begin{remark}\label{rem_3.10}
Non-$K$-maximal proper facets (faces of dimension $q-1$) of $\D^*$ correspond to the extreme directions of $\P$ by a similar duality relation, where the coupling function $\varphi$ has to be replaced by $\hat \varphi :\R^q\times\R^q\to \R, \quad \hat\varphi(y,y^*):=\varphi(y,y^*)+y^*_q$. The case $C=\R^q_+$, $c=(1,\dots,1)^T$ has been studied in \cite{Loehne11book} and the general case is obtained likewise using the generalized variants of \eqref{p1}, \eqref{d1}, \eqref{p2} and \eqref{d2} as defined below.
\end{remark} 

The following scalarization techniques are fundamental for the algorithms described in the next section. As mentioned in Remarks \ref{rem_3.7}, \ref{rem_3.9} and \ref{rem_3.10}, they can also be used to prove weak, strong and geometric duality. The {\em weighted sum} scalarized problem for a parameter vector $w \in C^+$ satisfying $c^T w = 1$ is
\begin{equation*}\tag{P$_1(w)$}\label{p1}
	\min w^T P x  \quad \text{ subject to } \;  Bx \geq b.
\end{equation*}
Its dual problem is
\begin{equation*} \tag{D$_1(w)$}\label{d1}
 \max b^T u \quad \text{ subject to } \; \left\{
     \begin{array}{rcl}
           B^T u &=& P^T w\\
               u &\geq& 0.
     \end{array}\right.
\end{equation*}
The {\em translative scalarization} (or scalarization by reference variable) is based on problem
\begin{equation*} \tag{P$_2(y)$} \label{p2}
\min z \quad \text{ subject to } \; \left\{
     \begin{array}{rcl}
           Bx &\geq& b\\
           Z^T P x &\leq& Z^T y + z \cdot Z^T c,
     \end{array}\right.
\end{equation*}
where $Z$ is the matrix whose columns are the generating directions of the dual cone $C^+$ of the ordering cone $C$. 
The dual program is
\begin{equation*}
\tag{$\bar {\rm D}_2(y)$} \label{d2a}
 \max b^T u - y^T Z v \quad \text{ subject to } \; \left\{
     \begin{array}{rcl}
           B^T u - P^T Z v&=& 0\\
                   c^T Z v&=& 1\\
            (u,v) &\geq& 0.
     \end{array}\right.
\end{equation*}
This problem can be equivalently expressed as
\begin{equation*}
\tag{D$_2(y)$}\label{d2}
 \max b^T u - y^T w \quad \text{ subject to } \; \left\{
     \begin{array}{rcl}
           B^T u - P^T w&=& 0\\
                   c^T w&=& 1\\
                   Y^T w&\geq& 0\\
            u &\geq& 0,
     \end{array}\right.
\end{equation*}
where $Y$ is the matrix of generating directions of the ordering cone $C$. The equivalence of \eqref{d2a} and \eqref{d2} is a consequence of the following assertion. For vectors $w \in \R^q$, we have
\[ Y^T w \geq 0 \iff \forall y \in C:\; y^T w \geq 0 \iff w \in C^+ \iff \exists v \geq 0:\; w = Z v.\]

\section{Benson's algorithm and its dual variant}\label{sec_alg}
\label{sec_Bensonalgo}

Benson \cite{Benson98a,Benson98} motivated his outer approximation algorithm by practical problems that typically involve a huge number of variables and constraints and just a few objective functions. He proposed three advantages of ``outcome set-based approaches'' in comparison to ``decision set-based approaches''. First, he observed that the set of minimal elements (in the outcome space $\R^q$) has a simpler structure than the set of minimizers (in the decision space $\R^n$) because, typically, $q \ll n$. This is beneficial for computational reasons but also for the decision maker. Second, in practice, decision makers prefer to base their decisions (at least in a first stage) on objectives rather than directly on a set of efficient decisions. Third, it is generic that many feasible points are mapped on a single image point which may lead to redundant calculations of ``little or no use to the decision maker'' \cite{Benson98}.

Comparing this motivation with the notions of the previous section, we see that the solution concepts are based on exactly the same motivation (but additionally there is a theoretical motivation, see \cite{HeyLoe11,Loehne11book}). It is therefore not surprising that the variants of Benson's algorithm presented here just compute solutions in the sense of the previous section.

The dual variant of the algorithm (based on geometric duality) has been established in \cite{EhrLoeSha07,EhrLoeSha12}. It was followed by approximating variants \cite{ShaEhr08,ShaEhr08-1} and by a generalization of the primal algorithm to convex problems \cite{EhrShaSch11}. Unbounded problems have been first treated in  \cite{Loehne11book}.
\begin{definition}
Problem \eqref{p} is said to be {\em bounded}, if
\[ \exists y \in \R^q:\quad P[S] \subset \cb{y} + C.\]
\end{definition}
The generalization from $\R^q_+$ to arbitrary pointed solid polyhedral convex cones is new in this article but has already been used in \cite{LoeRud11}. We will present simplified variants where only one linear program (rather than two or three) has to be solved in one iteration\footnote{This simplification was initiated by an idea of Kevin Webster. During a lecture in the Ph.D. course in spring 2011 at ORFE, Princeton University, where the classical variant of Benson's algorithm was introduced, he proposed a variant with the two LPs \eqref{p2} and \eqref{d2}. The advantage over the classical version is that \eqref{p2} and \eqref{d2} are dual to each other. All further improvements of this article are based on this idea.}.

The idea of the primal algorithm is to evaluate the upper image $\P = P[S] + C$ of problem \eqref{p} by constructing appropriate cutting planes. This leads to an iterative refinement of an outer approximation $\T \supset \P$ by a decreasing sequence of polyhedral supersets
\[ \T^0 \supset \T^1 \supset \dots \supset \T^k = \P. \]
Both an H-representation and a V-representation of the approximating supersets $\T^i$ are stored. The algorithm terminates after finitely many steps (say $k$ steps) when the outer approximation coincides with $\P$.

Unbounded problems are treated by a two-phase method. First, one solves the homogeneous problem \eqref{ph} (which is unbounded, too) and its dual problem
\begin{equation*}\label{dh}
\tag{D$^{*h}$} \text{ maximize } D^{*h}: \R^m \times \R^q \to \R^q \text{ with respect to } \leq_K \text{ over } T
\end{equation*}
with objective function
\[ D^{*h}:\R^m\times\R^q \to \R^q,\quad D^{*h}(u,w):=\of{w_1,...,w_{q-1},0}^T.\]
To this end, \eqref{ph} is transformed into an equivalent bounded problem
\begin{equation}\tag{P$^\eta$}\label{p_eta}
 \text{ minimize } P: \R^n\to\R^q \text{ with respect to } \le_C \text{ subject to } Bx \geq 0, \; \eta^T P x \leq 1,
\end{equation}
where $\eta \in \Int (\D^{*h}+K)$ with $c^T \eta=1$ ($\D^{*h}$ denotes the lower image of \eqref{dh}).

In the second phase, a primal and dual solution of the homogeneous problem \eqref{ph} are used to calculate a primal and dual solution of the original (inhomogeneous and unbounded) problem \eqref{p}. The two-phase method requires an algorithm that works whenever an H-representation of an initial outer approximation $\T^0 \supset \P$ with $\T^0_\infty = \P_\infty$ is known. If an H-representation of $\P_\infty$ is known, that is
\[\P_\infty=\cb{y \in \R^q \st (w^i)^T y \geq 0, i=1,\dots,r},\]
and if $\gamma_i$ denotes the optimal value of (P$_1(w^i)$) for $i=1,\dots,r$, then
\[\T^0=\cb{y \in \R^q \st (w^i)^T y \geq \gamma_i, i=1,\dots,r}\]
is the desired outer approximation of $\P$ satisfying $\T^0_\infty=\P_\infty$.

If problem \eqref{p} is bounded, we have $C=\P_\infty$, i.e., an H-representation of the ordering $C$ is required. Otherwise, whenever $\eqref{p}$ is feasible, the upper image $\P^h:=P[S^h]+C$ of the homogeneous problem \eqref{ph} coincides with $\P_\infty$. By geometric duality, a dual solution to \eqref{ph} yields an H-representation of $\P^h=\P_\infty$.

The idea of such an algorithm can be explained geometrically. Let $\T^0 \supset \P$ be an initial outer representation of $\P$, i.e., $\T^0_\infty = \P_\infty$. First, the vertices of $\T^0$ are computed from its H-representation. This can be realized by {\em vertex enumeration}, which is a standard method in Computational Geometry, see e.g. \cite{BarDobHuh96,BreFukMar98}. Secondly, for a vertex $t^0 \in \T^0$, problem (P$_2(t^0)$) is solved. Usually, LP solvers yield simultaneously a solution of both the primal and the dual problem. If the optimal value of (P$_2(t^0)$) is zero, then $t^0$ belongs to $\P$ and one proceeds with the next vertex of $\T^0$. If every vertex of $\T^0$ belongs to $\P$, we have $\T^0=\P$. Otherwise, for $t^0 \not\in \P$, a solution of (P$_2(t^0)$) yields a point $s^0 \in \bd\P \cap \Int \T^0$, see Proposition \ref{prop_1} below. The solution of the dual problem, defines a supporting hyperplane $H^0$ of $\P$ that contains $s^0$. The corresponding halfspace $H^0_+$ contains $\P$ but not $t^0$. An H-representation of an improved outer approximation $\T^1 := \T^0 \cap H^0_+$ is obtained immediately. This procedure is repeated until, after finitely many steps, $\T^k = \P$. A solution $(\bar S,\bar S^h)$ of \eqref{p} is obtained by collecting those points $x$ that arise during the procedure from a solution $(x,z)$ of (P$_2(t^i)$) with zero optimal value. In this case we have $t=Px$ for some vertex $t$ of $\T^i$. Hence $Px$ is a vertex of $\P$ which implies that $x$ is a minimizer for \eqref{p}. In the unbounded case, $\bar S^h$ contains directions that originate from a solution to the homogeneous problem.
A solution of the dual vector optimization problem \eqref{d} is obtained by collecting those dual solutions of (P$_2(t^i)$) with non-zero optimal value.

\begin{proposition}\label{prop_1}
Let $S \neq \emptyset$, $C\subseteq\R^q$ a solid pointed polyhedral cone and let $c \in \Int C$. For every $t \in \R^q$, there exist optimal solutions $(\bar x,\bar z)$ to $({\rm P}_2(t))$ and $(\bar u,\bar w)$ to $({\rm D}_2(t))$. Each solution $(\bar u, \bar w)$ to $({\rm D}_2(t))$ defines a supporting hyperplane $H:=\cb{y \in \R^q \st \bar w^T y = b^T \bar u}$ of $\P$ such that $s := t + \bar z\cdot c \in H\cap \P$. We have 
\[ t \not\in \P \iff \bar z > 0, \qquad t \in \wMin\P \iff \bar z = 0,\qquad t \in \Int \P \iff \bar z < 0.\]
\end{proposition}
\begin{proof} Fix $t\in \R^q$. Since $S\neq \emptyset$ and $c \in \Int C$, $({\rm P}_2(t))$ is feasible. Assuming  
$({\rm P}_2(t))$ is not bounded, we obtain $t + (z-n)c - P x \in C$ for all $n \in \N$. Dividing by $n$ and letting $n\to \infty$, we conclude $-c \in C$. As $c \in \Int C$, convexity of $C$ implies $0 \in \Int C$. Thus $C=\R^q$, a contradiction. Consequently, $({\rm P}_2(t))$ and, by duality, also   
$({\rm D}_2(t))$ have optimal solutions $(\bar x,\bar z)$ and $(\bar u, \bar w)$, respectively. Duality yields $b^T \bar u - t^T \bar w = \bar z$ and thus $s=t+ \bar z c$ belongs to $H$. Of course, $H$ is a hyperplane as the constraint $\bar w^T c =1$ of $({\rm D}_2(t))$ implies $\bar w \neq 0$. From $P \bar x \leq_C t + z\cdot c$ we conclude that $s=t+ \bar z c$ belongs to $\P$. For arbitrary $y \in \P$, there exists $x \in S$ such that $y \geq_C P x$. Hence $(x,0)$ is feasible for $({\rm P}_2(y))$. Weak duality between $({\rm P}_2(y))$ and $({\rm D}_2(y))$ implies that $y^T w \geq b^T u$ for every $(u,w)\in T$, in particular, $y^T \bar w \geq b^T \bar u$. Hence $H = \cb{y \in \R^q \st y^T \bar w = b^T \bar u}$ is a supporting hyperplane to $\P$. The remaining statements are now obvious, where the fact $\wMin \P = \bd \P$ can be used.
\end{proof}

\begin{proposition}\label{prop_2}
Every vertex of $\P$ is minimal.
\end{proposition}
\begin{proof}
Let $y \in \P = P[S] + C$ be not minimal for $\P$. Then there are $z \in \P$ and $k \in C\setminus \cb{0}$ such that $z = y - k$. The points $y - k$ and $y + k$ belong to $\P$ and we have $y = \frac{1}{2}(y - k) + \frac{1}{2}(y + k)$. Hence $y$ is not a vertex of $\P$.
\end{proof}

Two functions are used in the following algorithm. The function {\em dual()} computes a V-representation of an outer approximation $\T$ from an H-representation of $\T$, i.e., this function consists essentially of vertex enumeration. This H-representation of $\T$, however, is stored in a dual format, namely, as a V-representation of an inner approximation $\T^*$ of the lower image $\D^*$ of \eqref{d}, where 
\begin{equation}\label{tts}
\T^*=\cb{y^*\in \R^q \st \varphi(y,y^*) \geq 0, y \in \T}.
\end{equation}
The following duality relation holds.
\begin{proposition}\label{prop_tst}
If $\emptyset\neq\T\subsetneq\R^q$ is closed and convex and $\T_\infty\supseteq C$, then
\begin{equation}\label{tst}
\T=\cb{y \in \R^q \st \varphi(y,y^*) \geq 0, y^* \in \T^*}.
\end{equation}
\end{proposition}
\begin{proof} The inclusion $\subseteq$ is obvious. Assume that the inclusion $\supseteq$ does not hold, i.e., there exists $\bar y \in \R^q\setminus \T$ with $\varphi(\bar y,y^*)\geq 0$ for all $y^* \in \T^*$. Applying separation arguments, we get $\eta \in C^+\smz$ with $\eta ^T \bar y < \inf_{y \in \T} \eta^T y := \gamma$. By \eqref{ass_c}, we can assume $\eta^T c = 1$. Setting $\bar y^*:=(\eta_1,\dots,\eta_{q-1},\gamma)$, we get $w(\bar y^*)=\eta$ and $\varphi(y,\bar y^*)=w(\bar y^*)-\bar y^*_q=\eta^T y-\gamma$. For all $y \in \T$, we have $\eta^T y-\gamma \geq 0$, i.e., $\bar y^* \in \T^*$. But $\varphi(\bar y,\bar y^*)=\eta^T \bar y-\gamma< 0$, a contradiction.
\end{proof}

In the following algorithm, a V-representation of a polyhedron $\T$ (that contains no lines) is denoted by $(\T_{points},\T_{dirs})$, i.e., $\T = \conv \T_{points} + \cone \T_{dirs}$. We assume further that a V-representation returned by the function {\em dual()} is {\em minimal}, i.e., $\T_{points}$ consists of only vertices of $\T$ and $\T_{dirs}$ consists of only extreme directions of $\T$.

The function {\em solve()} returns an optimal solution $(x,z)$ of (P$_2(t)$) and an optimal solution $(u,w)$ of (D$_2(t)$). Since (D$_2(t)$) is, up to a substitution, the dual program of (P$_2(t)$), only one linear program has to be solved.

The variables in the following algorithm are arrays of vectors. By $\abs{A}$ we denote the number of vectors in an array $A$ and by $A[i]$ we refer to the $i$-th vector in $A$. The command {\em break} exits the inner-most loop.

\subsubsection*{Algorithm 1.}
 
Input:\\
\indent $B, b, P, Z$ (data of \eqref{p});\\
\indent a solution $(\cb{0},\bar S^h)$ to \eqref{ph};\\
\indent a solution $\bar T^h$ to \eqref{dh};

\noindent Output:\\
\indent $(\bar S,\bar S^h)$ ... a solution to \eqref{p};\\
\indent $\bar T$ ... a solution to \eqref{d};\\
\indent $(\T_{points},\T_{dirs})$ ... a V-representation of $\P$;\\
\indent $(\T^*_{points},\cb{-e^q})$ ... a V-representation of $\D^*$;

\noindent \\
\indent $\bar T \leftarrow \cbg{\ofg{\text{solve(D$_1$($w$))},w} \big|\; (u,w) \in \bar T^h}$;\\
\indent {\bf repeat}\\
\indent\indent $flag \leftarrow 0$;\\
\indent\indent $\bar S \leftarrow \emptyset$;\\
\indent\indent $\T^*_{points} \leftarrow \cb{D^*(u,w)\st (u,w) \in \bar T}$;\\
\indent\indent $(\T_{points},\T_{dirs}) \leftarrow \dual(\T^*_{points},\cb{-e^q})$;\\
\indent\indent {\bf for} $i=1$ {\bf to} $\abs{\T_{points}}$ {\bf do}\\
\indent\indent\indent $t \leftarrow \T_{points}[\,i\,]$;\\
\indent\indent\indent $(x,z,u,w)\leftarrow$ solve(P$_2$($t$)/D$_2$($t$));\\
\indent\indent\indent {\bf if} $z > 0$ {\bf then}\\
\indent\indent\indent\indent $\bar T \leftarrow \bar T\cup \cb{(u,w)}$;\\
\indent\indent\indent\indent $flag \leftarrow 1$;\\
\indent\indent\indent\indent break; (optional)\\
\indent\indent\indent {\bf else}\\
\indent\indent\indent\indent $\bar S \leftarrow \bar S \cup \cb{x}$;\\
\indent\indent\indent {\bf end if};\\
\indent\indent {\bf end for};\\
\indent {\bf until $flag = 0$};\\

\begin{theorem} \label{corr_fin_1}
 Let $S\neq \emptyset$, suppose $\P^h$ has a vertex and assume that the command
\[ (\T_{points},\T_{dirs}) \leftarrow \dual(\T^*_{points},\T^*_{dirs})\]
generates a minimal V-representation of $\T$ from a given V-representation of $\T^*$ according to \eqref{tts}. Then Algorithm 1 is correct and finite.
\end{theorem}
\begin{proof}
As $\bar T^h$ is non-empty (by the definition of a solution), we can choose some $(u,w)\in \bar T^h$. Then $D^{*h}(u,w)=(w_1,\dots,w_{q-1},0)$ is $K$-maximal in $\D^{*h}$. Hence $u$ solves the homogeneous variant (i.e., we set $b=0$) of \eqref{d1}. Consequently, (D$_1$($w$)) (for arbitrary $b$) is feasible. Since $S \neq \emptyset$, (P$_1$($w$)) is feasible, too. Thus, by linear programming duality, \eqref{d1} has a solution.

The set $\T^*:=\conv \T^*_{points} + \cone \cb{-e^q}$ is a non-empty subset of $\D^*$. Hence, by  Theorem \ref{th_wgd}, after calling the function {\em dual()}, $\T:=\conv \T_{points} + \cone \T_{dirs}$ is a superset of $\P$.

As $\bar T^h$ solves the dual of the homogeneous problem, we have $\T_\infty=\P_\infty=\P^h$, see \cite[Section 4.6]{Loehne11book} for more details. As $\P^h$ is assumed to have a vertex, $\T$ must have a vertex, hence the array $\T_{points}$ is non-empty. 

By Proposition \ref{prop_1}, solutions to (P$_2$($t$)) and (D$_2$($t$)) exist.
The vectors $x \in \bar S$ are minimizers of \eqref{p}. Indeed, $x$ is added to $\bar S$ only if $z=0$. In this case, we have $t \in \P$, where $t$ is a vertex of $\T\supseteq \P$ because, by assumption, $\T_{points}$ contains only vertices of $\T$. Hence $t$ is a vertex of $\P$ and, by Proposition \ref{prop_2}, $t$ is a minimizer for \eqref{p}.  

The algorithm terminates if all vertices of $\T$ belong to $\P$. Since $\P_\infty=\T_\infty$, we conclude $\P=\T$, i.e., $(\bar S,\bar S^h)$ is an infimizer of \eqref{p} and $(\T_{points},\T_{dirs})$ is a V-representation of $\P$.

A solution $(u,w)$ to (D$_2$($t$)) is always a maximizer of $\eqref{d}$, i.e., $\bar T$ consists of only maximizers. Since at termination $\T=\P$, Theorem \ref{th_fsd} implies $\T^*=\D^*$ and thus $\bar T$ is a supremizer for \eqref{d} and $(\T^*_{points},\cb{-e_q})$ is a V-representation of $\D^*$. 

Finally we show that the algorithm terminates after a finite number of steps.
The point $s^k:=t^k + z^k \cdot c$ computed in iteration $k$ (consider the `repeat' loop) by solving 
(P$_2$($t^k$)/D$_2$($t^k$)) belongs to $\Int \T^{k-1}$ whenever $z^k>0$. We have $ \T^k :=  \T^{k-1} \cap \{ y \in \R^q \st (w^k)^T y \geq (u^k)^T b\}$ and by Proposition \ref{prop_1} we know that $F:=\{ y \in \P \st (w^k)^T y = (u^k)^T b\}$ is a face of $\P$ with $s^k \in  F$, where $F \subset \bd  \T^k$. This means for the next iteration that $s^{k+1} \not \in F$ (because $s^{k+1} \in \Int \T^k$), and therefore $s^{k+1}$ belongs to another face of $\P$. Since $\P$ is polyhedral, it has a finite number of faces, hence the algorithm is finite.
\end{proof}

We now turn to the dual variant of Algorithm 1. An analogous construction is now applied to the lower image $\D^*$, i.e., a finite sequence of polyhedral sets
\[ \T^{*0} \supseteq \T^{*1} \supseteq ,\dots, \supseteq \T^{*k} = \D^*\]
is calculated. Using the upper image $\P^h$ (which is a polyhedral cone) of the homogeneous problem \eqref{ph}, we define the set
\[ \Delta := \cb{y^* \in \R^q \st w(y^*) \in (\P^h)^+}. \]
The counterpart of Proposition \ref{prop_1} is the following.

\begin{proposition}\label{prop_d1}
Let $S\neq \emptyset$ and $t^* \in \Delta$. For $w:=w(t^*)$, $({\rm P}_1(w))$ has a solution and for every such solution $\bar x$, $H^*(P \bar x)$ is a supporting hyperplane of $\D^*$ that contains 
\begin{equation}\label{eq_sst}
  s^*:=(t^*_1,\dots,t^*_{q-1},w^T P \bar x) \in \textstyle\Max_K \D^*.
\end{equation}
Moreover, we have 
\[ 
t^* \not\in \D^* \iff w^T P\bar x < t^*_q, \quad\qquad 
t^* \in \textstyle\Max_K \D^* \iff w^T P\bar x = t^*_q.
\]
\end{proposition}
\begin{proof} Since $t^* \in \Delta$, for all $k \in \P^h$, we have $w^T k \geq 0$. This means that the homogeneous variant of the linear program \eqref{p1} (i.e., we set $b=0$ in \eqref{p1}) is bounded (and feasible, as $0$ is feasible). Consequently, the dual program is feasible, even for arbitrary $b$, i.e., \eqref{d1} is feasible. On the other hand, \eqref{p1} is feasible, since we assumed $S \neq \emptyset$. Altogether this implies that both \eqref{p1} and \eqref{d1} have optimal solutions denoted, respectively, by $\bar x$ and $\bar u$. Strong duality implies $w^TP\bar x = b^T \bar u$. Thus, \eqref{eq_sst} holds. We have $s^* \in H^*(P\bar x)$ because this can be written as 
$w(s^*)^T P \bar x = s^*_q$ where we have $w=w(t^*)=w(s^*)$. Together with Theorem \ref{th_wgd}, we obtain that $H^*(P\bar x)$ is a supporting hyperplane of $\D^*$ that contains $s^*$.
The remaining statements are now obvious.
\end{proof}

The following consequence of Proposition \ref{prop_tst} is useful to characterize the condition $t^* \in \Delta$.

\begin{corollary}\label{cor_del}
Let the assumptions of Proposition \ref{prop_tst} be satisfied. Then, $w(y^*) \in (\T_\infty)^+$ for all $y^* \in \T^*$.
\end{corollary}
\begin{proof} Assuming the contrary, there is $\bar y^*\in \T^*$ and $k \in \T_\infty$ with $w(y^*)^T k < 0$. Let $\bar y \in \T$. For sufficiently large $\lambda>0$, using \eqref{eq_phiw}, we obtain $\varphi(\bar y+\lambda k,\bar y^*)< 0$, which contradicts Proposition \ref{prop_tst}.
\end{proof}

The following dual algorithm has the same input and output as Algorithm 1. Similar functions are used. The function {\em dual()} computes a V-representation of an outer approximation $\T^*$ of $\D^*$ from a V-representation of an inner approximation $\T$ of $\P$. In contrast to Algorithm 1, it is not necessary that {\em dual()} returns a minimal V-representation. The recession cone of sets $\T^*$ occurring in the algorithm is known, in fact, we always have $\T^*_\infty=-K=\R_+\cb{-e^q}$. Therefore we denote the return of the function {\em dual()} by $(\T^*_{points},\sim)$ indicating that the second return value (the array containing the extreme directions of $\T^*$) is not used.

The function {\em solve()} returns an optimal solution $x$ of (P$_1(w)$) and an optimal solution $u$ of (D$_1(w)$). Again, only one linear program has to be solved.

\subsubsection*{Algorithm 2.}
 
Input:\\
\indent $B, b, P, Y$ (data of Problem \eqref{p});\\
\indent a solution $(\cb{0},\bar S^h)$ to \eqref{ph};\\
\indent a solution $\bar T^h$ to \eqref{dh};

\noindent Output:\\
\indent $(\bar S,\bar S^h)$ is a solution to \eqref{p};\\
\indent $\bar T$ is a solution to \eqref{d};\\
\indent $(\T_{points},\T_{dirs})$ ... a V-representation of $\P$;\\
\indent $(\T^*_{points},\cb{-e_q})$ ... a V-representation of $\D^*$;

\noindent \\
\indent $\T_{dirs} \leftarrow \cb{Px \st x \in \bar S^h}\cup \cb{y \st y \text{ is a column of } Y}$;\\
\indent $\bar w  \leftarrow \text{ mean} \cb{ w \st (u,w) \in \bar T^h}$;\\
\indent $\bar S \leftarrow \cb{\text{solve(P}_1\text{(}\bar w\text{))}}$;\\
\indent {\bf repeat}\\
\indent\indent $flag \leftarrow 0$;\\
\indent\indent $\bar T \leftarrow  \emptyset$;\\
\indent\indent $\T_{points} \leftarrow \cb{Px \st x \in \bar S}$\\
\indent\indent $(\T^*_{points},\sim) \leftarrow \dual(\T_{points},\T_{dirs})$;\\
\indent\indent {\bf for} $i=1$ {\bf to} $\abs{\T^*_{points}}$ {\bf do}\\
\indent\indent\indent $t^* \leftarrow \T^*_{points}[i]$;\\
\indent\indent\indent $w \leftarrow w(t^*)$;\\
\indent\indent\indent $(x,u)\leftarrow$ solve(P$_1$($w$)/D$_1$($w$));\\
\indent\indent\indent {\bf if} $t^*_q - b^T u > 0$ {\bf then} \\
\indent\indent\indent\indent $\bar S \leftarrow \bar S \cup \cb{x}$;\\
\indent\indent\indent\indent $flag \leftarrow 1$;\\
\indent\indent\indent\indent break; (optional)\\
\indent\indent\indent {\bf else}\\
\indent\indent\indent\indent $\bar T \leftarrow \bar T \cup \cb{(u,w)}$;\\
\indent\indent\indent {\bf end if};\\
\indent\indent {\bf end for};\\
\indent {\bf until $flag=0$};\\
\indent delete points $x \in \bar S$ whenever $Px$ is not a vertex of $\T$ ;\\

\begin{remark}
The last line in the algorithm is easy to realize, for instance, by computing a minimal V-representation using the command
\[(\T_{points},\T_{dirs}) \leftarrow \dual(\T^*_{points},\T^*_{dirs})\]
from Algorithm 1 by standard vertex enumeration methods. Then one has to test if for $x \in \bar S$, $Px$ belongs to $\T_{points}$, if not, $x$ is deleted from $\bar S$. In particular, it is not necessary to solve a linear program. 
\end{remark}

\begin{theorem} \label{corr_fin_2}
 Let $S\neq \emptyset$ and assume that $\P^h$ has a vertex. Then, Algorithm 2 is correct and finite.
\end{theorem}
\begin{proof} By similar arguments as in the proof of Theorem \ref{corr_fin_1} one can show that P$_1$($\bar w$)) has a solution.

The set $\T:=\conv \T_{points} + \cone \T_{dirs}$ is a subset of $\P$. Hence, by  Theorem \ref{th_wgd}, after calling the function {\em dual()}, $\T^*:=\conv \T^*_{points} + \cone \T^*_{dirs}$ is a superset of $\D^*$. Since $\T \neq \R^q$, $c \in \Int C$ and $\T_\infty \supseteq C$, we have $\cone \T^*_\infty = \R_+ \cdot \cb{-e^q}$, i.e., we can set $\T^*_{dirs}=\cb{-e^q}$ and we know that $\T^*_{points} \neq \emptyset$.

The array $(\cb{0},\T_{dirs})$ provides a V-representation of $\P^h$, i.e., $\T_\infty=\P^h$. Corollary \ref{cor_del} yields that $\T^*_{points} \subset \Delta$. Hence, by Proposition \ref{prop_d1}, solutions to (P$_1$($w$)) and (D$_1$($w$)) exist. It can be easily shown that the vectors $(u,w) \in \bar T$ are maximizers of \eqref{d}, see also \cite[Lemma 4.51]{Loehne11book}.  

The algorithm terminates, if $\T^*_{points} \subseteq \D^*$. Since $\D^*_\infty=\T^*_\infty=\R_+\cb{-e^q}$, we conclude $\D^*=\T^*$, i.e., $\bar T$ is a supremizer of \eqref{d} and $(\T^*_{points},\cb{-e^q})$ is a V-representation of $\D^*$. Since at termination $\T^*=\D^*$, Theorem \ref{th_fsd} implies $\T=\P$. Thus $(\T_{points},\T_{dir})$ is a (not necessarily minimal) V-representation of $\P$.
A solution $x$ to (P$_1$($w$)) is in general not a minimizer for $\eqref{p}$ (but only ``weakly efficient'', compare e.g. \cite[Theorem 4.1]{Loehne11book}). Therefore, in the last line of the algorithm, $x$ is deleted from $\bar S$, whenever $Px$ is not a vertex of $\T$. According to Proposition \ref{prop_2}, the remaining set $\bar S$ consists of only minimizers. It is non-empty because, by assumption, $\P^h$ has a vertex and hence $\T=\P$ must have a vertex. As non-vertex points are redundant in a V-representation of a set which has a vertex, the property of $\bar S$ being an infimizer for \eqref{p} is maintained by deleting the non-minimizers in $\bar S$.

Finally we show that the algorithm terminates after finitely many steps. We consider the `repeat' loop in iteration $k$. We set $w=w^k$ and $t^*={t^{*k}}$ and denote the solutions of (P$_1$($w^k$)) and (D$_1$($w^k$)) by $x^k$ and $u^k$, respectively. The point $s^{*k}:=t^{*k} + z^{*k} \cdot \cb{-e^q}$, where $z^{*k}:=(t^{*k}_q - b^T u^k)$ belongs to $\T^{*(k-1)}\setminus \Max_K \T^{*(k-1)}$ whenever $z^{*k}>0$. We have $ \T^{*k} :=  \T^{*(k-1)} \cap \{ y^* \in \R^q \st \varphi(Px^k,y^*)\geq 0\}$ and by Proposition \ref{prop_d1} we know that $F^*:=\{ y^* \in \D^* \,\st\varphi\ofg{Px^k,y^*}= 0\}$ is a face of $\D^*$ with $s^{*k} \in  F^*$. Likewise to \cite[Lemma 4.48]{Loehne11book}, we see that $F^*\subseteq \Max_K \T^{*k}$. This means for the next iteration that $s^{*(k+1)} \not \in F^*$ (because $s^{*(k+1)} \in \T^{*k}\setminus\Max_K \T^{*k}$), and therefore $s^{*(k+1)}$ belongs to another face of $\D^*$. Since $\D^*$ is polyhedral, it has a finite number of faces, hence the algorithm is finite.
\end{proof}

Let us summarize the two-phase method for solving unbounded problems. We consider an arbitrary linear vector optimization problem, where we only assume that $C$ is a solid pointed polyhedral cone. We fix some $c$ according to \eqref{ass_c}, which is always possible in the way described after \eqref{ass_c}. In phase 1, we first try to compute some $\eta \in \Int(\D^* + K)$ with $\eta^T c = 1$. This can be realized by Algorithm 3 in \cite[Section 5.5]{Loehne11book}, where the set $T$ has to be adapted to the more general setting of this article. Note that $c$ has a different meaning in \cite[Section 5.5]{Loehne11book}. The first LP solved by the mentioned algorithm is
\[ \min 0^T w + 0^T u \quad\text{ s.t. }\quad (u,w) \in T.\] 
If this linear program is infeasible, then \eqref{d} is infeasible. Otherwise one obtains either some $\eta \in \Int(\D^* +K)$ or the algorithm indicates that $\Int(\D^* +K)$ is empty. In the latter case, we know that $\P^h$ has no vertex. This means that $\P$, if non-empty, contains a line. This case has not been treated so far. Since $c_q=1$ according to \eqref{ass_c}, the condition $\eta^T c = 1$ can be always realized by an appropriate choice of $\eta_q$. 

Next, we solve \eqref{p_eta} by either Algorithm 1 or Algorithm 2. Since \eqref{p_eta} is bounded, a solution of the primal and dual homogeneous problem of \eqref{p_eta} can be easily obtained. However, this is not necessary as the $u$-components of $(u,w)\in \bar T^h$ are not used in Algorithms 1 and 2. Therefore we can use  
\[ \bar S^h = \emptyset \qquad \text{and}\qquad\bar T^h = \cb{\of{0,\frac{z}{z^T c}} \bigg|\; z \text{ is a column of } Z }\]
as an input of Algorithm 1 or 2 to solve \eqref{p_eta}, compare also \cite[Theorem 5.20]{Loehne11book}. Let $(\bar S_\eta,\bar S^h_{\eta})$ be a solution of $\eqref{p_eta}$ and let $\bar T_\eta$ be a solution of the dual problem of \eqref{p_eta}. Then, a solution of \eqref{ph} is obtained by setting
\[ (\bar S,\bar S^h) := (\cb{0},\bar S_\eta\smz),\]
compare \cite[Theorem 5.23]{Loehne11book}. A solution $\bar T^h$ of $\eqref{dh}$ can be obtained from $\bar\T_{\eta}$ but again only the $w$-components are required by Algorithms 1 or 2 in phase 2.
 As a consequence of \cite[Theorem 5.25]{Loehne11book}, we can use 
\[ \bar T^h:=\cb{\of{0,w(y^*)}\st y^* \in \T^*_{points},\; y^*_q=0}\]
as an input of Algorithm 1 or 2 in the second phase, where $\T^*_{points}$ is the result from the first phase, i.e., $(\T^*_{points},\cb{-e^q})$ is a V-representation of the lower image $\D^{*\eta}$ of the dual problem of \eqref{p_eta}. In the second phase, the first LP to be solved in Algorithm 1 is \eqref{d1}. If \eqref{d1} turns out to be unbounded, we know that \eqref{p} is infeasible. Likewise, if the first LP in Algorithm 2, namely (P$_1$($\bar w$)), is infeasible, we know that \eqref{p} is infeasible. Otherwise, according to Theorems \ref{corr_fin_1} and \ref{corr_fin_2}, solutions of \eqref{p} and \eqref{d} are computed. 

\begin{remark}\label{rem_eps}
In practice the condition $z>0$ in Algorithm 1 is replaced by $z>\varepsilon$ for some $\varepsilon>0$. Assume that the results of the first phase are always exact. Then, in the second phase, Algorithm 1 yields an $\varepsilon$-solution of \eqref{p} in the sense that in Definition \ref{def_sol} the finite infimizer is replaced by a finite $\varepsilon$-infimizer, i.e., condition \eqref{eq_inf_att} is replaced by
\begin{equation}\label{eq_inf_att_eps}
\conv P[\bar S]  + \cone P[\bar S^h] + C -\varepsilon\cb{c} \supseteq P[S] + C.
\end{equation}
Taking into account that (using the assumption $c_q=1$ in \eqref{ass_c})
\[ \varphi(y-\varepsilon c,y^*)=\varphi(y,y^*+\varepsilon e^q),\]
we see that Algorithm 1 yields an $\varepsilon$-solution of \eqref{d} in the sense that in Definition \ref{def_sold} a finite supremizer is replaced by a finite $\varepsilon$-supremizer, i.e., condition \eqref{eq_sup_att} is replaced by 
\begin{equation}\label{eq_sup_att_eps}
 \conv D^*[\bar T] - K +\varepsilon\cb{e^q} \supseteq D^*[T] -K.
\end{equation}
Likewise, in Algorithm 2 the condition $t^*_q - b^T u > 0$ is replaced by $t^*_q - b^T u > \varepsilon$ for some $\varepsilon > 0$. Consequently, Algorithm 2 yields an $\varepsilon$-solution of \eqref{d}. It also yields an
$\varepsilon$-infimizer of \eqref{p}, but in general not an $\varepsilon$-solution of \eqref{p}. The reason is that the last line in Algorithm 2 only works for the exact algorithm.

Note further that an $\varepsilon$-solution $(\bar S,\bar S^h)$ of \eqref{p} refers to an inner and an outer approximation of the upper image $\P$ in the sense that 
\[ \conv P[\bar S]  + \cone P[\bar S^h] + C \subseteq \P \subseteq \conv P[\bar S]  + \cone P[\bar S^h] + C -\varepsilon\cb{c}.\]
Likewise, an $\varepsilon$-solution $\bar T$ of \eqref{d} refers to an inner and an outer approximation of the lower image $\D^*$ in the sense that 
\[ \conv D^*[\bar T] - K \subseteq \D^* \subseteq \conv D^*[\bar T] - K +\varepsilon\cb{e^q}.\]
Note that the approximation error of the classical variant of Benson's algorithm and its dual variant has been studied in \cite{ShaEhr08,ShaEhr08-1}.
\end{remark}

\section{Computation of polyhedral set-valued risk measures}

Set-valued risk measures evaluate the risk of multi-variate random portfolios $X \colon \Omega \to \R^d$ the components $X_i\of{\omega}$ of which represent the number of units of the $i$-th asset in the portfolio, $i = 1, \ldots, d$. If transaction costs are present, such risk measures are more appropriate than real-valued functions, which always represent a complete risk preference and thus cannot account for incomparable portfolios.

The theory of set-valued risk measures was initiated in \cite{JMT04} and systematically developed in \cite{HamHey10} and \cite{HamHeyRud11}. We refer the reader to these references for further motivation and information. Here, we restrict ourself to the case of finite probability spaces and the question how the values of a set-valued risk measure can be computed. It will turn out that this leads to problems of type (P), hence one can apply the algorithm presented in Section~\ref{sec_Bensonalgo}.

The basic idea is as follows. The value of a set-valued risk measure at some random future portfolio $X$ consists of initial deterministic portfolios $u$ which can be given as deposits for the `risky payoff' $X$, thus making the overall position `risky payoff plus deposit' a non-risky one. It usually is not possible to use all assets as deposits, but rather a small subset including cash in a few currencies, bonds, gold or similar risk-free or low-risk assets. These `eligible' assets are assumed to span the linear subspace $M$ of $\R^d$ with $1 \leq \dim M = m \leq d$. A typical example, already used in \cite{JMT04}, is $\R^m \times \cb{0}^{d-m}$, i.e. the first $m$ assets are eligible. 

Let $\of{\Omega, P}$ be a finite probability space and $N \geq 2$ be the number of elements in $\Omega = \cb{\omega_1, \ldots, \omega_N}$. We assume $p_n = P\of{\cb{\omega_n}} > 0$ for all $n \in \cb{1, \ldots, N}$. The space of all multi-variate random variables $X \colon \Omega \to \R^d$ is denoted by $L^0_d\of{\Omega, P}$.  A random variable  $X \in L^0_d\of{\Omega, P}$ can be identified with an element $\hat x \in \R^{dN}$ through
\[
\hat{x} = \of{X_1\of{\omega_1}, \ldots, X_d\of{\omega_1}, X_1\of{\omega_2}, \ldots, X_d\of{\omega_N}}^T \in \R^{dN}
\]
and vice versa. Thus, the function $T \colon L^0_d\of{\Omega, P} \to \R^{dN}$ defined by $TX = \hat x$ is a linear bijection. If $A \subseteq L^0_d\of{\Omega, P}$ then we set $\hat A = \cb{\hat x \in \R^{dN} \mid \hat x = TX, \; X \in A}$.

Let $K_0 \subseteq \R^d$ be a finitely generated convex cone satisfying $\R^d_+ \subseteq K_0 \neq \R^d$. Such a `solvency' cone models the market at initial time. We set $K_0 ^M = K_0  \cap M$ and $\mathcal P\of{M, K_0^M} = \cb{D \subseteq M \mid D = D + K_0^M}$. A risk measure is a function $R \colon L^0_d\of{\Omega, P} \to \mathcal P\of{M, K_0^M}$ satisfying
\begin{equation}
\label{EqMTranslative}
\forall u \in M, \; \forall X \in L^0_d\of{\Omega, P} \colon R\of{X + u\One} = R\of{X} - u
\end{equation}
where $\One \colon \Omega \to \R$ with $\One\of{\omega} = 1$ for all $\omega \in \Omega$ is the uni-variate random variable with constant value 1. 

With $R$, we associate a risk measure $\hat R \colon \R^{dN} \to \mathcal P\of{M, K_0^M}$ by means of $\hat R\of{\hat x} = R\of{X}$ for $\hat x = TX$. Consequently, $\hat R$ satisfies
\begin{equation}
\label{EqMTranslativeDiscrete}
\forall u \in M, \; \forall \hat x \in \R^{dN} \colon \hat R\of{\hat x + \hat I_d u} =\hat R\of{\hat x} - u
\end{equation}
where
\[
\hat{I}_d = \left(
          \begin{array}{c}
           I_d \\
           \vdots  \\
           I_d
           \end{array}
           \right) \in \R^{dN \times d} \quad \text{and} \quad
I_d =  \left(
        \begin{array}{cccc}
          1 & 0 & \ldots & 0 \\
          0 & 1 & \ldots & 0 \\
           \vdots & &  \ddots & \vdots \\
           0 & \ldots & 0 & 1 \\
        \end{array}
      \right) \in \R^{d \times d}.
\]

The most common way to generate a risk measure is by means of a set $A \subseteq L^0_d\of{\Omega, P}$ of random variables which are considered to be `acceptable' by the decision maker. The value of a risk measure generated by $A$ then consists of all deterministic (available at time $0$) portfolios $u \in M \subseteq \R^d$ which, when added to the uncertain future position $X$, make the overall position acceptable. Thus,
\[
R_A\of{X} = \cb{u \in M \mid X + u\One \in A}.
\]
This functions indeed satisfies \eqref{EqMTranslative}. Correspondingly, 
\[
R_{\hat A}\of{\hat x} = \cb{u \in M \mid \hat x + \hat I_d u \in \hat A}
\]
satisfies \eqref{EqMTranslativeDiscrete}, and we have $R_{\hat A} = \widehat{R_A}$. Vice versa, with risk measures $R \colon L^0_d\of{\Omega, P} \to \mathcal P\of{M, K_0^M}$  and $\hat R \colon \R^{dN} \to \mathcal P\of{M, K_0^M}$ we associate the sets
\[
A_R = \cb{X \in L^0_d\of{\Omega, P} \mid 0 \in R\of{X}} \; \text{and} \;
	A_{\hat R} = \cb{\hat x \in \R^{dN} \mid 0 \in \hat R\of{\hat x}},
\]
respectively. A basic fact about risk measures is a one-to-one correspondence between closed acceptance sets $A \subseteq L^0_d\of{\Omega, P}$ which satisfy $A + K_0^M\One \subseteq A$ and risk measures $R \colon L^0_d\of{\Omega, P} \to \mathcal P\of{M, K_0^M}$ with a closed graph by means of the above formulas. In particular, the relationships $\hat A  = A _{R_{\hat A}}$ and $\hat R = R_{A_{\hat R}}$ hold true. See \cite{HamHeyRud11} for further details.

A risk measure $R$ is called polyhedral if the associated risk measure $\hat R$ is polyhedral, i.e., if
\[
\gr \hat R = \cb{\of{\hat x, u} \in \R^{dN} \times M \mid u \in \hat R\of{\hat x}}
\]
is a polyhedral subset of $\R^{dN} \times \R^d$. The one-to-one correspondence between risk measures and their acceptance sets extends to the polyhedral case: $\hat A \subseteq \R^{dN}$ is polyhedral if, and only if, $R_{\hat A}$ is polyhedral, and $\hat R$ is polyhedral if and only if $\hat A_R$ is polyhedral.

The above discussion leads to the following conclusion.

\begin{remark}
\label{RemPRMvsBenson} Since each polyhedral risk measure $\hat R$ has the representation
\[
\hat R\of{\hat x} =  \cb{u \in M \mid \hat x + \hat I_d u \in A_{\hat R}}
\]
where $A_{\hat R}$ is a polyhedral set, the set $\hat R\of{\hat x}$ is the upper image of a linear vector optimization problem. Indeed, if $A_{\hat R} \subseteq \R^{dN}$ has the H-representation $A_{\hat R} = \cb{\hat y \in \R^{dN} \mid \hat B \hat y \geq \hat b}$ where $\hat B$ and $\hat b$ are matrices of appropriate dimension then
\[
\hat R\of{\hat x} =  \cb{u \in M \mid \hat B\of{\hat x + \hat I_d u} \geq \hat b} = \cb{u \in M \mid \hat B \hat I_d u \geq \hat b - \hat B\hat x}.
\]
Let $P \in \R^{d \times m}$ be a matrix with column vectors $\mu^1, \ldots, \mu^m$ forming a basis of $M$ and define $B = \hat B \hat I_dP$, $b = \hat b - \hat B\hat x$. Then, observing that $\hat R\of{\hat x} + K_0^M = \hat R\of{\hat x}$ and substituting $u = Pz$ we obtain that $\hat R\of{\hat x}$ is the upper image of the problem 
\[
 \text{ minimize } P \colon \R^m \to M \text{ with respect to } \le_{K_0^M} \text{ subject to } Bz \geq b.
\]
\end{remark}

However, this is just a theoretical result since in practice life is not as straightforward: Usually, the constraints describing $\hat A$ involve a large number of auxiliary variables, and $u$ is given as a linear function of those (see Example~\ref{exAVARtheory} below). Therefore, the algorithm presented in Section~\ref{sec_Bensonalgo} is an appropriate tool to compute the values of a polyhedral set-valued risk measure because the dimension of the pre-image space usually is much greater than the dimension of the image space which is $\dim M = m \leq d$. Compare Examples~\ref{ex_WCtheory} and \ref{exAVARtheory} below.

In the following, we will discuss two examples which will be used for the numerical computations reported in Section~\ref{section_numeric}.

\begin{example}
\label{ex_WCtheory}
In worst case, the regulator/decision maker only accepts positions with non-negative components. Thus, the acceptance set  is $A = \of{L^0_d}_+$ which is the set of all component-wise non-negative random variables. The market extension of the worst case risk measure, i.e. when trading is allowed, is related to the set of superhedging portfolios, see \cite{LoeRud11}. Its acceptance set in a one-period market is $A = \of{L^0_d}_+ +K_0\One+L^0_d\of{K_T}$ where the cone $K_0$ and the random cone $K_T$ model market conditions with a potential bid ask price spread at initial and terminal time, respectively, and $L^0_d\of{K_T} = \cb{X \in L^0_d \mid \forall \omega \in \Omega \colon X\of{\omega} \in K_T\of{\omega}}$.  The cones $K_T\of{\omega}$ are also finitely generated convex cones satisfying $\R^d_+ \subseteq K_T\of{\omega} \neq \R^d$ for all $\omega \in \Omega$.

The market extension of the worst case risk measure still is very conservative since a payoff $X$ is acceptable only if there is a trading strategy such that its result, when added to $X$, is non-negative in all components in all possible scenarios, even those with a very small probability. Therefore, we introduce a `relaxed' variant  as follows. 
 
Let $G \subseteq \R^d$ be a finitely generated convex cone with $\R^d_+ \subseteq G \neq \R^d$ and consider the following acceptance set
\[
A^{RWC} = \cb{X \in L^0_d\of{\Omega, P} \mid \forall \omega \in \Omega \colon X\of{\omega}  \in \of{-\epsilon + \R^d_+} \cap G}+K_0\One+L^0_d\of{K_T}
\]
where $\epsilon \in \R^d$ such that $\epsilon_i \geq 0$ for all $i \in \cb{1, \ldots, d}$. Compared to the `true' worst case risk measure, the set $\of{L^0_d}_+$ is replaced by `a little' bigger set. 

Thus, payoffs with `small' negative components may still be considered acceptable, and the size of the risk related to such payoffs is controlled by $\epsilon$ and $G$. The cone $G$ may serve as a conservative estimate of a market model which the regulator/supervisor thinks is robust enough to cover most market scenarios. For example, $G$ can be chosen such that $P(G\subseteq K_T)\geq 1-\alpha$ for some significance level $\alpha\in[0,1]$. Then, the probability of a loss is bounded by $\alpha$, and a potential loss (in physical units) is bounded by $\epsilon$. Note that in the scalar case $d=m=1$, the relaxed worst case risk measure reduces to the scalar worst case risk measure.
The market extension of the relaxed worst case risk measure is given by 
\begin{align}
\label{WCMar}
 RWC\of{X} 
\nonumber  = & \big\{u \in M \mid  z \in K_0, \; Z \in L^0_d\of{K_T}, \; \forall \omega \in \Omega \colon
\\           & \;\;X\of{\omega} -z - Z\of{\omega} + u \in \of{-\epsilon + \R^d_+} \cap G\big\}
\end{align}
and can be seen as a relaxation of the superhedging set and thus as a good deal price bound of $-X$, see Example~\ref{ex45} for details. 
It is polyhedral (convex) as $A^{RWC}$ is polyhedral (convex), but not sublinear. This is a new feature since the classical worst case risk measure is always sublinear.

In order to describe $RWC\of{\hat x}$, let $g^1, \ldots, g^L$ be the generating vectors of the cone $G$ and let $\hat G$ be the $dN \times LN$ matrix which contains $N$ blocks on its diagonal, where each block consists of the matrix with $g^1, \ldots, g^L$ as columns. Then
\[
\forall n \in \cb{1, \ldots, N} \colon X\of{\omega_n} \in G \quad \Leftrightarrow \quad \exists \gamma \in \R^{LN}_+ \colon \hat x =  \hat G \gamma.
\]
Similarly, let $\hat K_T$ be the $dN \times J$ matrix which contains $N$ blocks on its diagonal, where the first block consists of the matrix with the generating vectors of $K_T(\omega_1)$ as columns, the second block contains the generating vectors of $K_T(\omega_2)$ and so forth to the last block with  the generating vectors of $K_T(\omega_N)$. $J$ is the sum of the number of generating vectors of all $K_T$'s. Let $\hat K_0$ denote the matrix containing the $I$ generating vectors of $K_0$ as columns. Then,
\begin{align*}
RWC\of{\hat x} = \{&Pz \mid \hat x + \hat I_d\of{Pz + \epsilon-\hat K_0r}- \hat K_Ts\in \R^{dN}_+,\\
   & \hat I_d\of{ Pz-\hat K_0r}- \hat K_Ts - \hat G \gamma = -\hat x, 
 z \in \R^m, \; \gamma \in \R^{LN}_+, \; r\in \R^{I}_+, \; s \in \R^{J}_+\},
\end{align*}
with ordering cone $K_0^M$. Thus, the dimension of the pre-image space is $m + LN+I+J$ whereas the dimension of the image space is just $m$.
\end{example}

\begin{example} 
\label{exAVARtheory}
The following set-valued function is a generalization of the scalar average value at risk (see  \cite[p. 210]{FS11}) which is probably the most important and most studied example of a sublinear coherent measure of risk as introduced by \cite{ADEH99}. Let $\alpha \in (0,1]^d$. Define for $X \in L^0_d$
\begin{align}
\label{AVaRMar}
 AV@R_\alpha\of{X} 
\nonumber  = &\big\{\diag\of{\alpha}^{-1}E\sqb{Z} - z \mid
    Z \in \of{L^0_d}_+, \\
     &\;\; X + Z - z\One \in K_0\One + L^0_d\of{K_T}, \; z \in \R^d\big\} \cap M
\end{align}
where $\diag\of{\alpha}^{-1}$ is the inverse of the diagonal matrix with the components of $\alpha$ on its main diagonal and zero elsewhere, and the cones $K_0, K_T\of{\omega}$ modeling the market conditions are as described above. Therefore this risk measure is also called the market extension of a simpler `regulator' version, see \cite{HamelRudloffYankova12}. We also refer to this paper for further motivation, interpretation and more details. It is immediately clear that $AV@R_\alpha$ is not given in the form of $R_A$ above, but it is a polyhedral convex (even sublinear) risk measure. 

Its `hat' variant can be derived as follows. Replace $Z$ by $\hat z \in \R^{dN}$ and introduce auxiliary variables which admit to write $\hat x + \hat z - \hat I_d z$ as non-negative linear combination of the generating vectors of the cones $K_0$ and $K_T\of{\omega_n}$.  
Transform the objective into matrix form and get
\[
AV@R_\alpha\of{\hat x} = \cb{Px \mid Bx \geq b} + K_0^M,
\]
with appropriate matrices $B$, $b$ and $P$ where $K_0^M = K_0 \cap M$ as before. The dimension of the pre-image space is $d(N+1)+I+J$ whereas the dimension of the image space is just $m$. It is worth mentioning that in the scalar case (i.e., without transaction costs) Rockafellar and Uryasev observed that the AV@R can be computed by solving a linear optimization problem, see \cite{RockafellarUryasev00}. 
\end{example}

\section{Numerical examples}
\label{section_numeric}

The algorithms have been implemented with MATLAB using the GNU Linear Programming Kit (GLPK) to solve the LPs and the CDDLIB package \cite{BreFukMar98} for vertex enumeration. The graphics have been generated by JavaView\footnote{by Konrad Polthier, http://www.javaview.de} and OpenOffice (Figure \ref{fig_5}). By a straightforward extension of the above results we can also solve linear vector optimization problems with constraints of the form
\begin{equation}\label{eq_constr}
 a \leq Bx \leq b \qquad lb\leq x \leq ub,
\end{equation}
where the components of $a,b,lb,ub$ belong to $\R \cup\cb{-\infty,+\infty}$.
All examples were computed on a MacBook Pro with 2.26 GHz clock and 8 GB memory.  We made use of the fact that all the LPs have a very similar form. This means that the matrix $B$ does not need to be changed during the algorithm (except one line because $\eta$ is not yet known at the beginning). This allows us to initialize LPs by appropriate basis solution of LPs solved in previous steps (warm starts). 

In the following examples we provide tables with a few computational data, such as the total time and the number of LPs solved (\# LPs). Note that we compute an $\varepsilon$-solution $(\bar S,\bar S^h)$ of \eqref{p} and an  $\varepsilon$-solution $\bar T$ of \eqref{d}, compare Remark \ref{rem_eps}. We provide the cardinality $|\,.\,|$ of the sets $\bar S$, $\bar S^h$ and $\bar T$. Recall that we have $|S^h|=0$ whenever the problem is bounded. 
Note that $|\bar S|$ and $|\bar T|$ `correlate' to the number of, respectively, vertices and facets of $\P$ (but do not need to coincide exactly). One reason for possible differences is degeneracy as discussed in \cite[Section 5.6]{Loehne11book}, another one is numerical inaccuracy.

Further we denote by $t_{max}$ the maximum time used to solve one LP and by $t_{aver}$ the average time to solve one LP. The quotient $t_{max}/t_{aver}$ indicates the impact of using warm starts. 
We start with two numerical examples from the literature. 

\begin{figure}[ht]
\begin{center}
\includegraphics[scale=0.58]{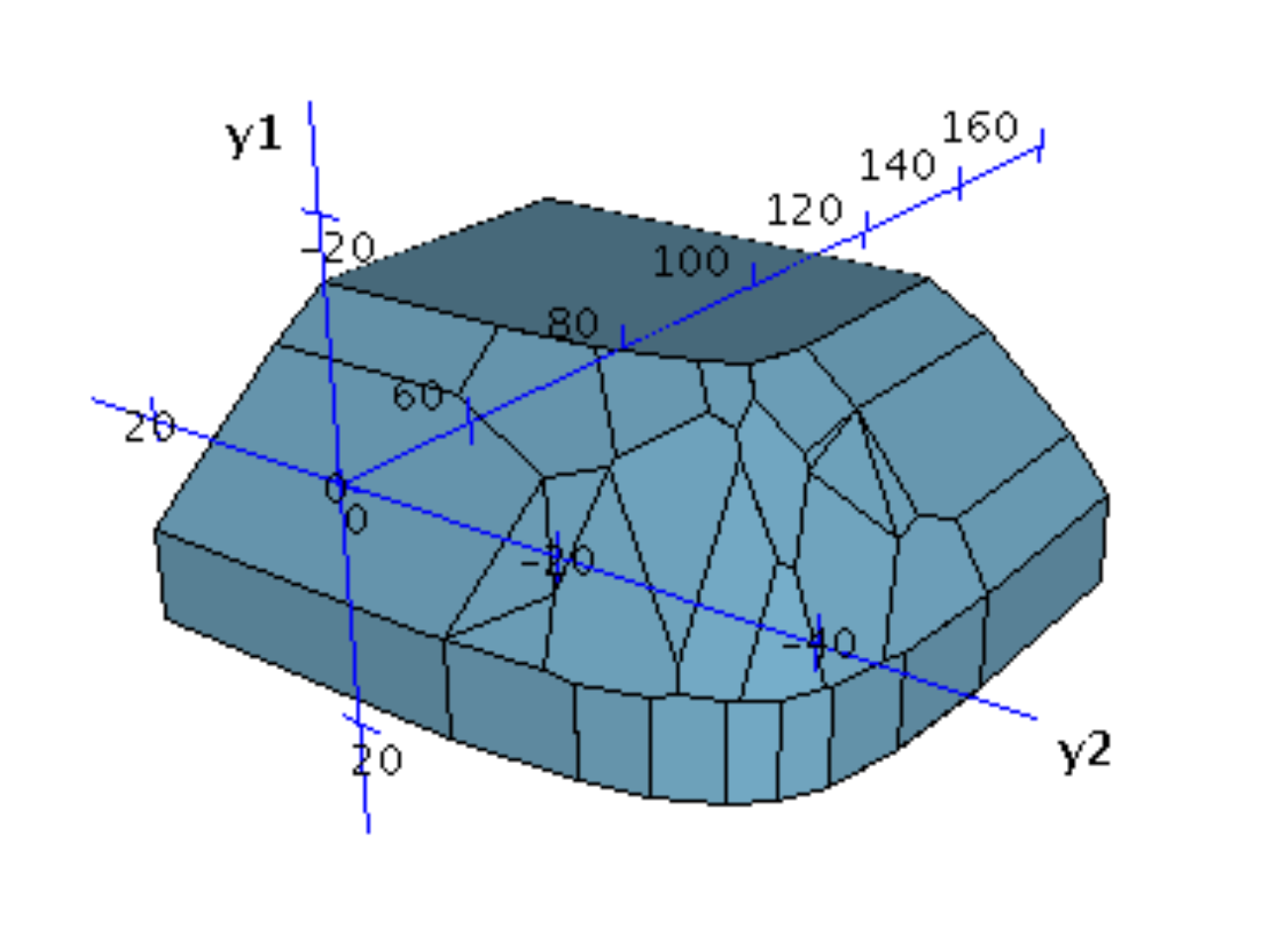}
\includegraphics[scale=0.58]{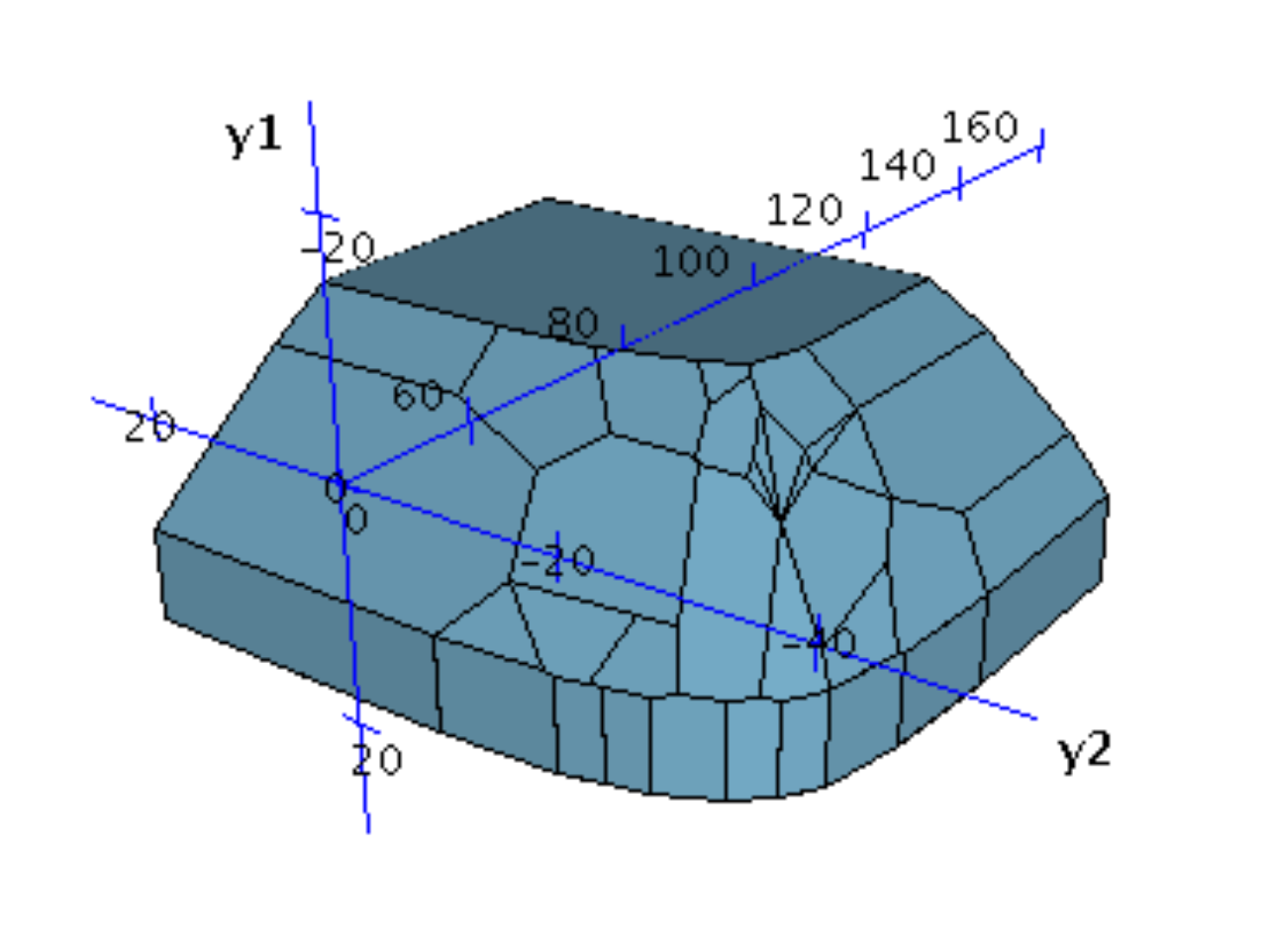}
\includegraphics[scale=0.58]{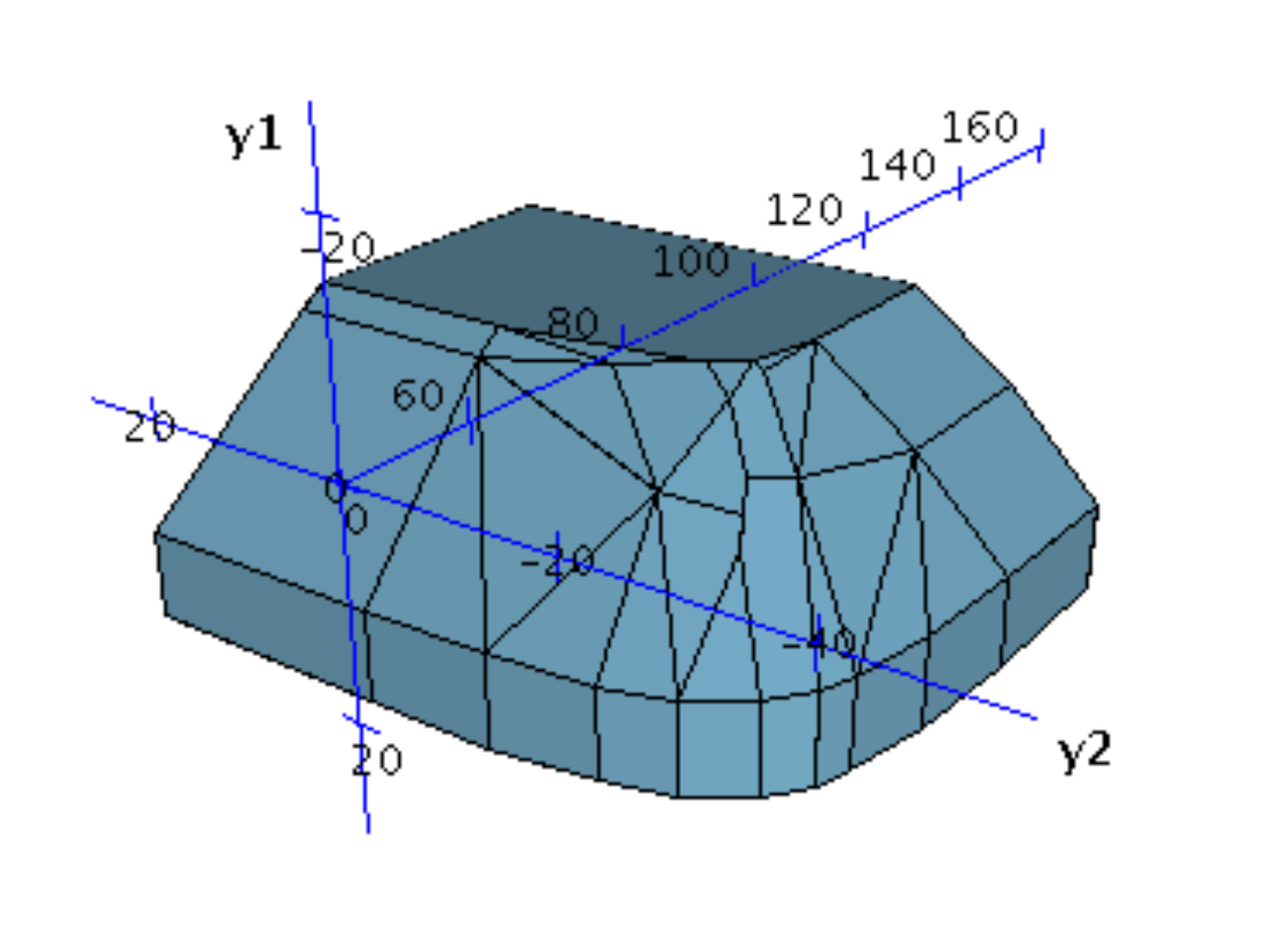}
\includegraphics[scale=0.58]{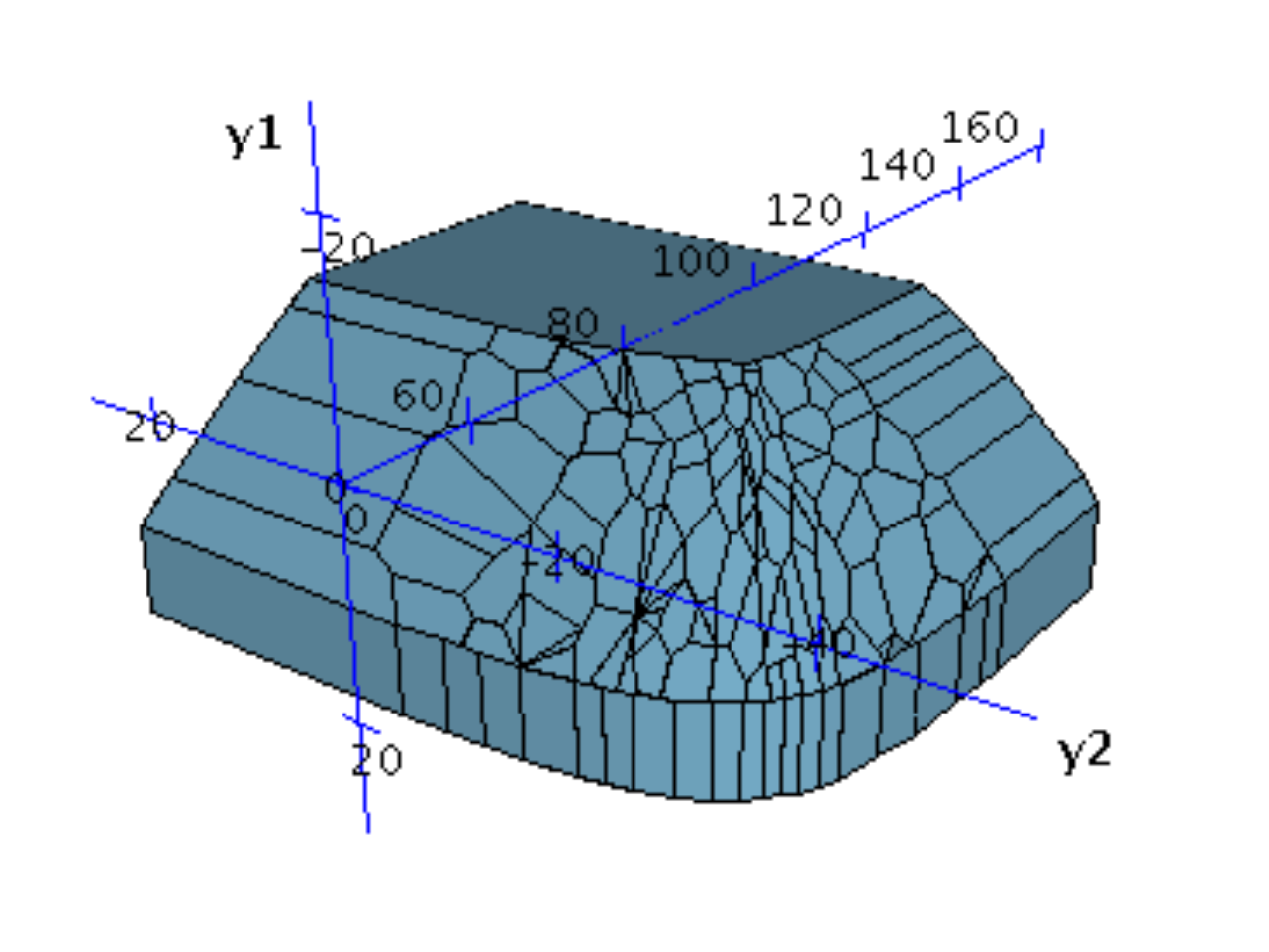}
\end{center}
\caption{Illustration of the upper image $\P$ in Example \ref{ex07}; top left: primal algorithm, variant `break', $\varepsilon=0.3$; top right: primal algorithm, variant `no break', $\varepsilon=0.3$; bottom left: dual algorithm, variant `break', $\varepsilon=0.3$; bottom right: primal algorithm, variant `break', $\varepsilon=0.05$.}
\label{fig1}
\end{figure}

\begin{example} \label{ex07} Shao and Ehrgott \cite{ShaEhr08} used extended variants of Benson's algorithm to solve linear vector optimization problems occurring in radio therapy treatment planning. We compute Example (PL) in \cite{ShaEhr08} which has three objectives and a matrix $B$ of size $1211 \times 1143$ with $153\,930$ nonzeros. The ordering cone is $C=\R^3_+$, and the problem is known to be bounded, which means that the first phase of our algorithms as well as the computation of $\eta$ can be skipped. Further we set $c=(1,1,1)^T$.

The following table shows some results obtained by Algorithm 1. The second column in the table concerns the optional break command in the algorithm. One can observe that more LPs have to be solved when the break command is disabled. On the other hand, less vertex enumerations are required. This explains why the variant `no break' is becoming faster than the `break' variant when $\varepsilon>0$ is small enough.

\begin{center} \begin{tabular}{|c|c|c|c|c|c|c|c|}
\hline $\varepsilon$ &  variant & total time& $|\bar S|$ & $|\bar T|$ & \# LPs & $t_{max}$ & $t_{max}/t_{aver}$ \\
\hline $0.3$         & break    & 47 secs   & 46       & 29        & 75     & 0.84 secs & 1.8                 \\
\hline $0.1$         & break    & 91 secs   & 104      & 61        & 165    & 0.87 secs & 2.0                 \\
\hline $0.05$        & break    & 144 secs  & 176      & 94        & 270    & 0.86 secs & 2.0                 \\
\hline $0.005$       & break    & 1596 secs & 1456     & 597       & 2053   & 0.84 secs & 1.9                 \\
\hline $0.3$         & no break & 54 secs   & 54       & 34        & 88     & 0.85 secs & 1.8                 \\
\hline $0.1$         & no break & 114 secs  & 134      & 78        & 212    & 0.84 secs & 1.9                 \\
\hline $0.05$        & no break & 205 secs  & 264      & 129       & 393    & 0.85 secs & 1.9                 \\
\hline $0.005$       & no break & 1411 secs & 1945     & 804       & 2749   & 0.84 secs & 1.9                 \\
\hline \end{tabular}
\end{center}

Although we need less computational time than in \cite{ShaEhr08}, it is difficult to compare the results as we use a faster computer, a different (open source) LP solver, and we utilize warm starts. Moreover, in \cite{ShaEhr08} an online vertex enumeration method is used, which is preferable if the number of vertices and facets of $\P$ is large. Furthermore, our method yields the same approximation error as in \cite{ShaEhr08} by less vertices and facets of $\P$. See Figure \ref{fig1} for an illustration of part of the results.
\end{example}

\begin{figure}[ht]
\begin{center}
\includegraphics[scale=0.53]{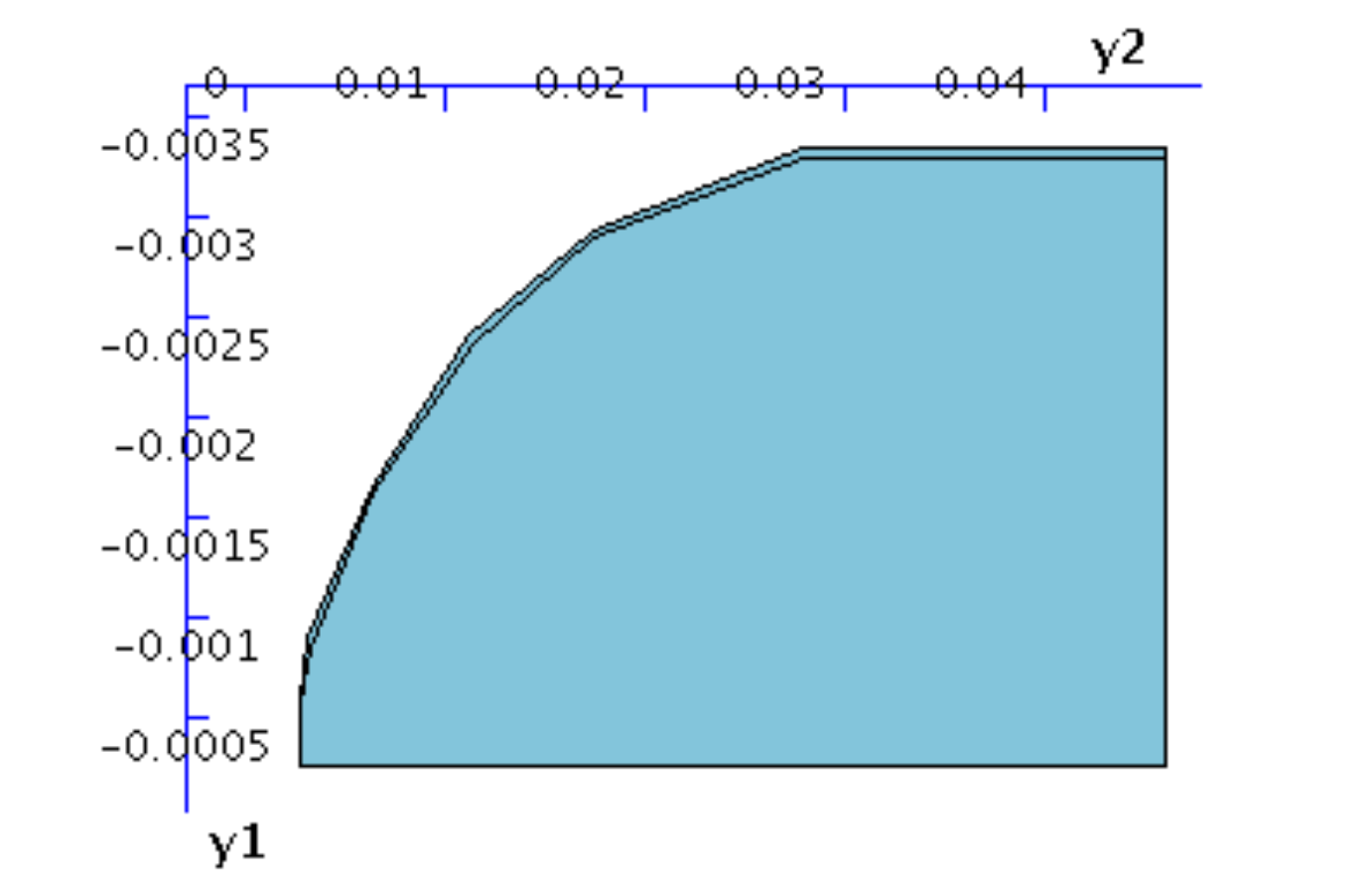}
\includegraphics[scale=0.53]{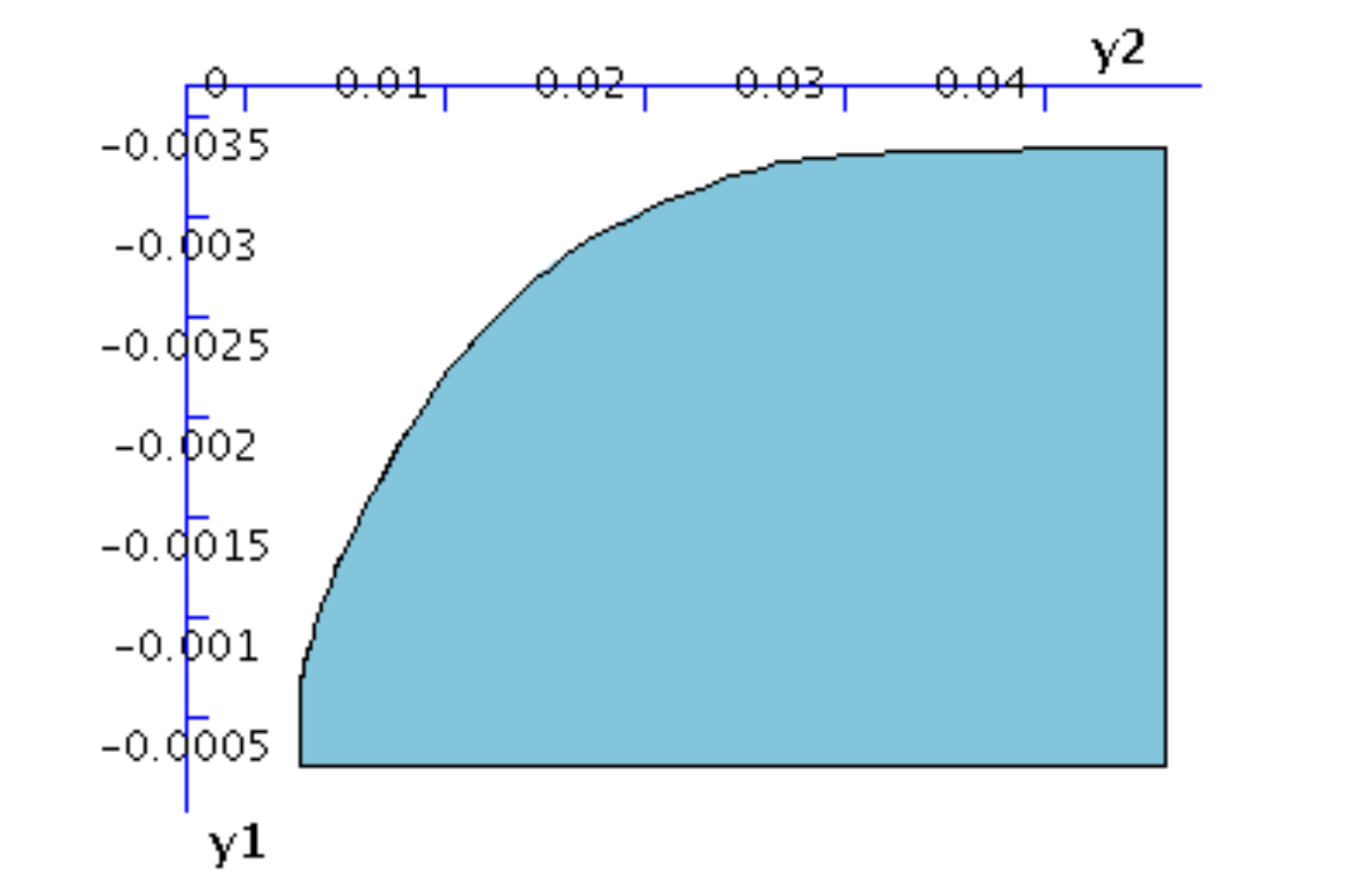}
\end{center}
\caption{Illustration of the upper image $\P$ in Example \ref{ex06} for $\varepsilon/\norm{c}=10^{-4}$ (inner and outer approximation, left) and for $\varepsilon/\norm{c}=10^{-6}$ (right).}
\label{fig_2}
\end{figure}

\begin{example}\label{ex06} Ruszczy\'nski and Vanderbei \cite{RusVan03} developed a specialized parametric method for computing all minimizers of bi-criteria problems.
Using intermediate results of the parametric simplex method, they solved in \cite{RusVan03}, for instance, a mean-risk model with a dense matrix $B$ of size $6161 \times 3799$ having $4\,435\,919$ nonzero entries. They pointed out that computing all the 5017 minimizers takes only a little more time than solving one single LP. As this problem is known to be bounded, we can skip in our algorithms the first phase as well as the computation of $\eta$. For $c=(1,1)^T$, our primal algorithm yields approximate solutions as shown in the following table.\noindent
\begin{center}
 \begin{tabular}{|c|c|c|c|c|c|c|}
\hline $\varepsilon$ & total time &$|\bar S|$&$|\bar T|$& \# LPs & $t_{max}$ & $t_{max}/t_{aver}$ \\
\hline $\sqrt{2}\cdot 10^{-4}$     &   946 secs & 6        & 7        & 13     & 347 secs & 4.4              \\
\hline $\sqrt{2}\cdot 10^{-5}$     &  1648 secs & 22       & 23       & 45     & 304 secs & 8.9              \\
\hline $\sqrt{2}\cdot 10^{-6}$     &  3085 secs & 62       & 63       & 125    & 310 secs & 14.1              \\
\hline \end{tabular}
\end{center}
We see that a `good' approximation with $\varepsilon/\norm{c}=10^{-6}$ can be obtained in about ten times the time required to solve a single LP. This means our general method needs much more time for an $\varepsilon$-solution (compare Remark \ref{rem_eps}) than the parametric method for bounded bi-criteria problems in \cite{RusVan03} needs for the exact solution. On the one hand, approximating solutions are often sufficient for a decision maker in practice, compare Figure \ref{fig_2}. On the other hand, we think that the ideas of the algorithm by Ruszczy\'nski and Vanderbei are promising for further improvements of Benson type algorithms for arbitrary linear vector optimization problems.
\end{example} 

The following three numerical examples refer to Example \ref{exAVARtheory} in the previous section.

\begin{figure}[ht]
\begin{center}
\includegraphics[scale=0.55]{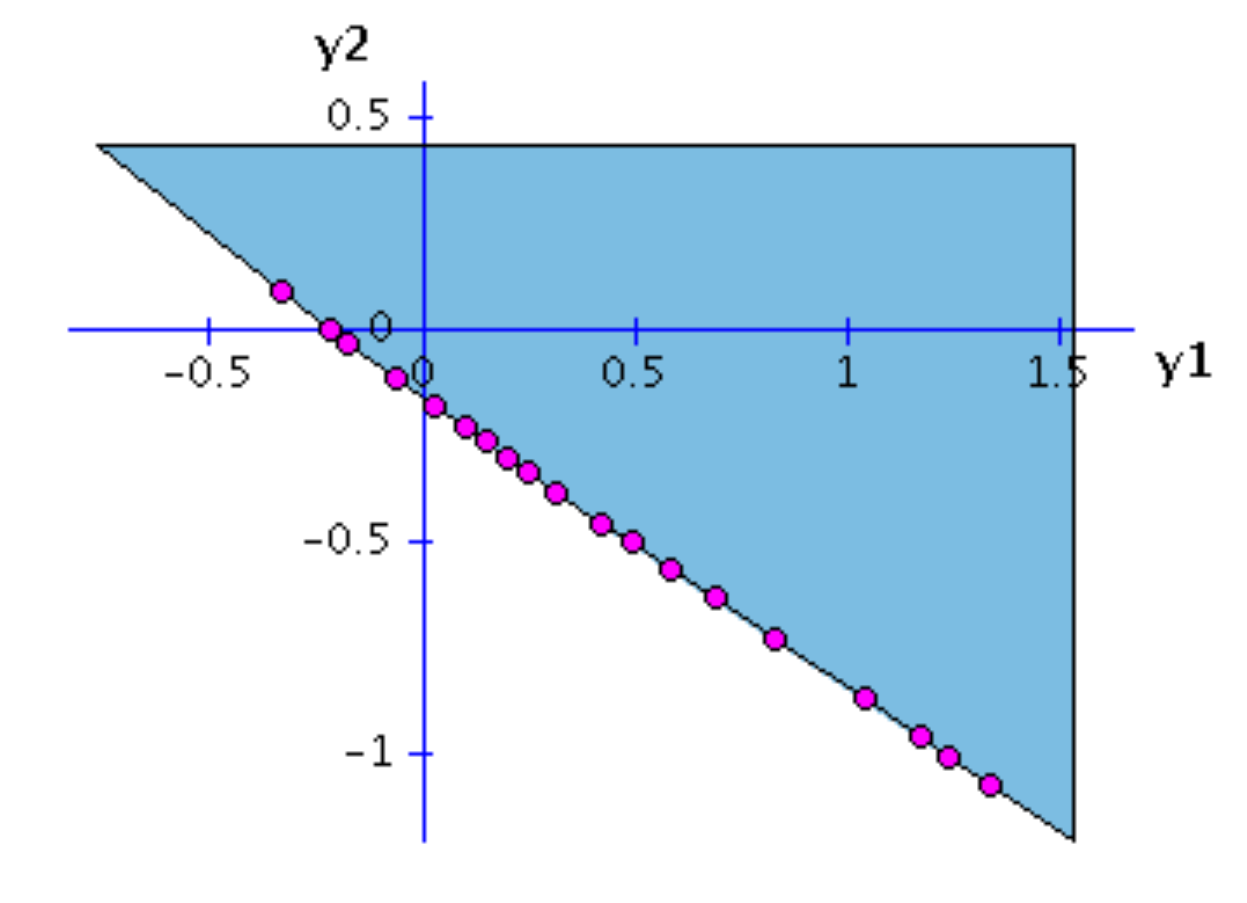}
\includegraphics[scale=0.55]{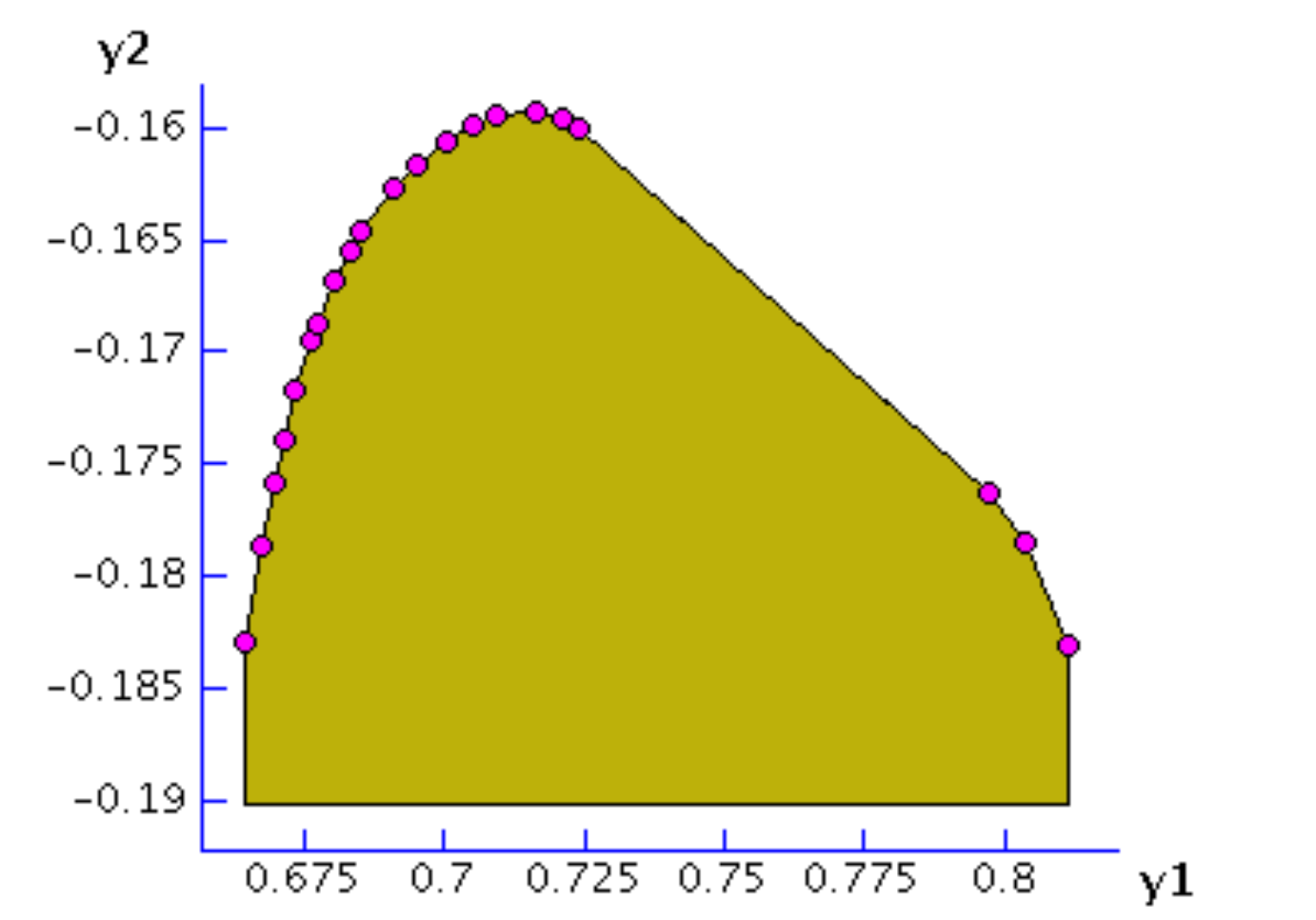}
\end{center}
\caption{Illustration of the upper image $\P$ (left) and the lower image $\D^*$ (right) in Example \ref{AVAR1} for $\varepsilon=10^{-4}$.}
\label{fig_3}
\end{figure}

\begin{example}
\label{AVAR1}
Let us consider $d=12$ assets, the first one is a risk-free USD bond with annual interest rate $5\%$. Given is the vector of today's asset prices, the vector of the expected returns and a covariance matrix for the other $11$  correlated risky assets denoted in USD. Then, one can set up a one-period tree for the asset prices $S_T$ with time horizon $T=1$ year as in \cite{KorMue09} to reflect the drift and covariance structure. The resulting number of scenarios is $N=2^{d-1}=2048$.
We consider proportional transaction costs for the bond to be $\lambda_0=3\%$ and for the first risky assets (usually another currency) to be $\lambda_1=7\%$, the second risky asset to be $\lambda_2=5\%$ and all other risky assets to be $1\%$. Then, the bid and ask prices of the assets are $(S_t^a)_i=(1+\lambda_i)(S_t)_i$ and $(S_t^b)_i=(1-\lambda_i)(S_t)_i$ for $i=0,\dots,11$ and $t\in\{0,T\}$.
Furthermore, let us assume an exchange between any two risky assets can not be made directly, only via cash in USD by selling one asset and buying the other. Since the risk-free bond has strictly positive transaction costs $\lambda_0$, the cones $K_0$ and $K_T(\omega_n)$ for $n=1,\dots,N$ have $d(d-1)=132$ generating vectors each. Thus $I=132$ and $J=270\,336$.

We want to evaluate the risk of an outperformance option with physical delivery and maturity $T$. This option gives the right to buy the asset that performed best out of a basket of assets at a given strike price.
Let the strike be $K=(1+\lambda_1)(S_0)_1$. To normalize to today's prices, let a vector $g$ be defined as $(S_0^a)_1=g_i(S_0^a)_i$ for $i\in\{1,\dots,11\}$. The payoff $X$ of the option is $-K$ in the risk free asset, $g_i$ units of asset $i$ for the smallest $i$ satisfying $g_i(S_T^a)^i=\max_{j\in\{1,...,11\}}(g_j(S_T^a)^j)\geq K$ and zero in the other assets. If $\max_{j\in\{1,...,11\}}(g_j(S_T^a)^j)< K$ the payoff is the zero vector.

Let us calculate $AV@R_\alpha(X)$ as described in Example~\ref{exAVARtheory} with significance levels \[\alpha=(0.1, 0.08, 0.09, 0.1, 0.05, 0.05, 0.04, 0.05, 0.03, 0.04, 0.03, 0.04 )^T.\] As the space of eligible assets we choose the space spanned by the first and the second asset, i.e. $M = \R^2 \times \{0\}^{10}$.
Formula \eqref{AVaRMar} leads to a linear vector optimization problem with 2 objectives and constraints of the form \eqref{eq_constr} where the matrix $B$ is of size $24\,586 \times 295\,056$. $B$ is sparse having $1\,150\,986$ nonzero entries. The ordering cone is $K_0^M$, which is strictly larger than $\R^2_+$ and is generated by $2$ vectors. The vertices of $AV@R_\alpha(X)$ are minimal deposits in the bond and the second asset that compensate for the risk of $X$ measured by $AV@R_\alpha$. The following table shows some computational data of the primal algorithm.
\noindent
\begin{center} \begin{tabular}{|c|c|c|c|c|c|c|c|}
\hline $\varepsilon$ & total time &$|\bar S|$&$|\bar S^h|$&$|\bar T|$& \# LPs & $t_{max}$ & $t_{max}/t_{aver}$ \\
\hline $10^{-4}$     &  3529 secs & 20       & 1          & 21       & 46     & 592 secs & 8.4                 \\
\hline $10^{-5}$     &  4716 secs & 47       & 1          & 48       & 100    & 671 secs & 17.1                \\
\hline $10^{-6}$     &  7905 secs & 122      & 1          & 123      & 253    & 449 secs & 22.0                \\
\hline \end{tabular}
\end{center}
We can see that the problem is unbounded. In Figure \ref{fig_3} the upper image $\P$ and the lower image $\D^*$ for $\varepsilon=10^{-4}$ are shown. We observe that the vertices of $\P$ are almost on a line and the lower image $\D^*$ is more suitable to illustrate the example. 
\end{example}

\begin{figure}[hpt]
\begin{center}
\includegraphics[scale=0.62]{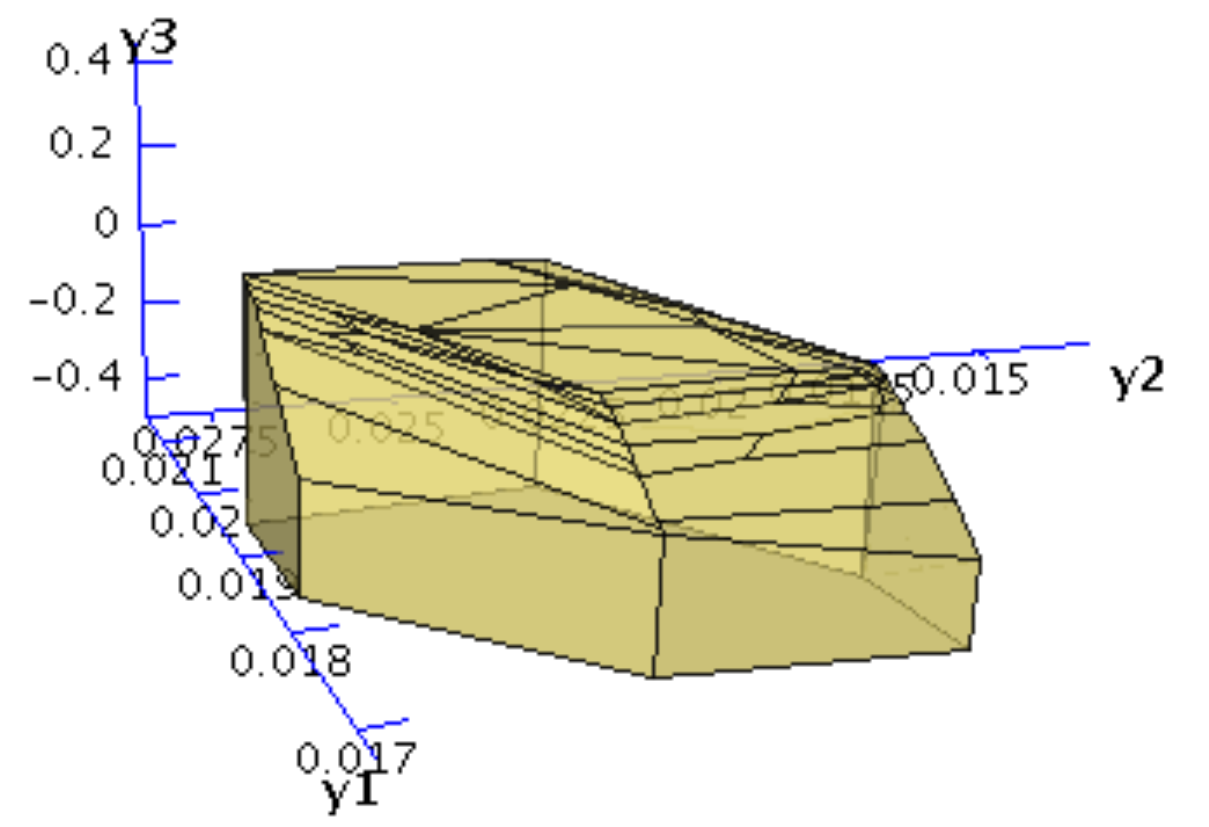}
\includegraphics[scale=0.62]{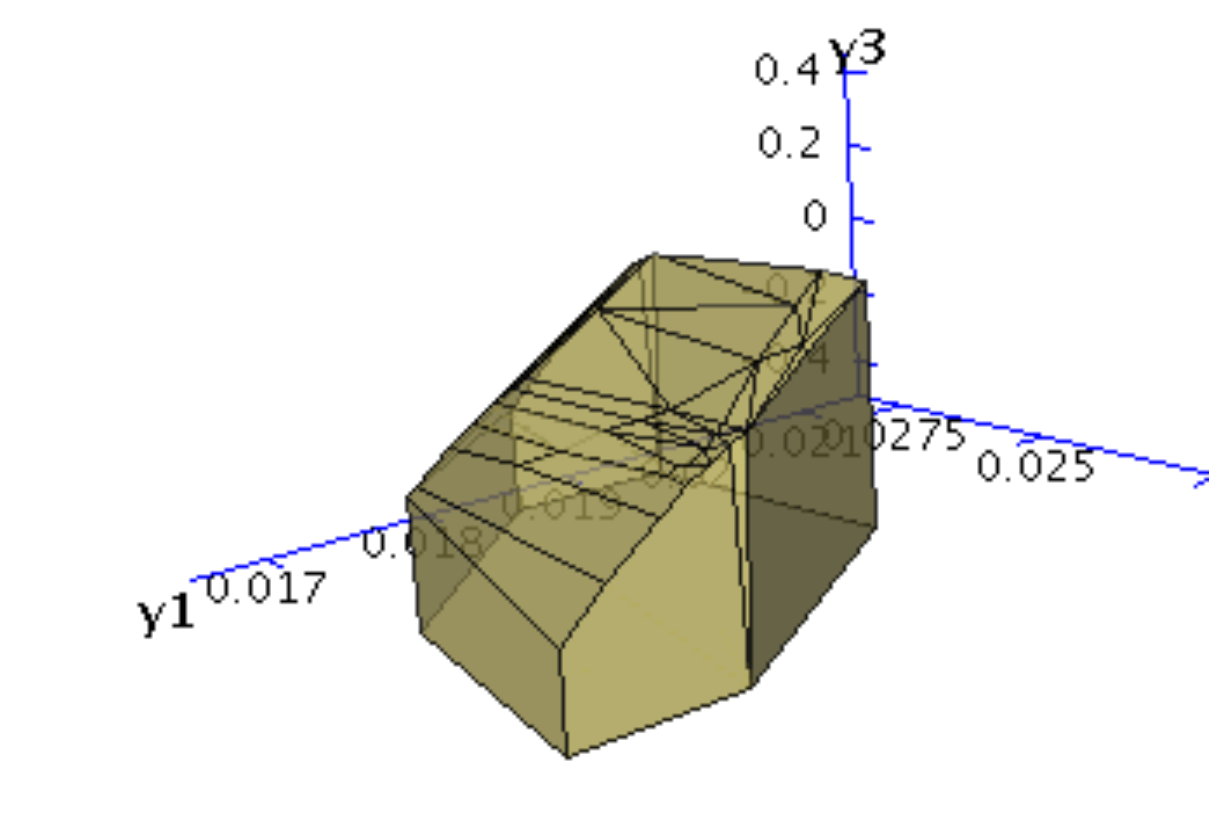}
\end{center}
\caption{The lower image $\D^*$ in Example \ref{ex39} for $\varepsilon=10^{-3}$, two different view points.}
\label{fig_4}
\end{figure}

\begin{example}
\label{AVAR2}\label{ex39}
Now consider $d=11$ assets with a given correlation structure and all other input parameters as for the first $11$ assets in Example~\ref{AVAR1} above. We have $N=2^{d-1}=1024$ and the number of generating vectors of each cone $K_0$ and $K_T(\omega_n)$ for $n=1,\dots,N$ is $d(d-1)=110$.
Consider a basket call option with physical delivery and strike price $K=\sum_{i=1}^{10} (S_0)_i$. If at maturity ($T=1$ year) the value of the basket of risky assets is greater or equal to the strike, i.e., $\sum_{i=1}^{10} (S_T)_i\geq K$, then one would exercise the option and buy the $10$ risky assets at strike $K$ by delivering $K$ times the bond, i.e. $X=(-K,1,\dots,1)^T$ in this case. If the value is less, the payoff vector is the zero vector. As the space of eligible assets we choose the space spanned by the first three assets, i.e. $M = \R^3 \times \{0\}^{8}$. The ordering cone is $K_0^M$, which is strictly larger than $\R^3_+$ and generated by $6$ vectors. 
The linear vector optimization problem has 3 objectives and a matrix of size $11\,272\times124\,025$, which is sparse having $481\,288$ nonzero entries. For $\varepsilon=10^{-3}$, the computational time of the primal algorithm was 1748 seconds. The result is illustrated in Figure \ref{fig_4}. As the upper image $\P$ is difficult to visualize (a polyhedron containing no lines but being `close' to a halfspace) we only provide the lower image $\D^*$.
\end{example}

\begin{figure}[h]
\begin{center}
\includegraphics[scale=0.38]{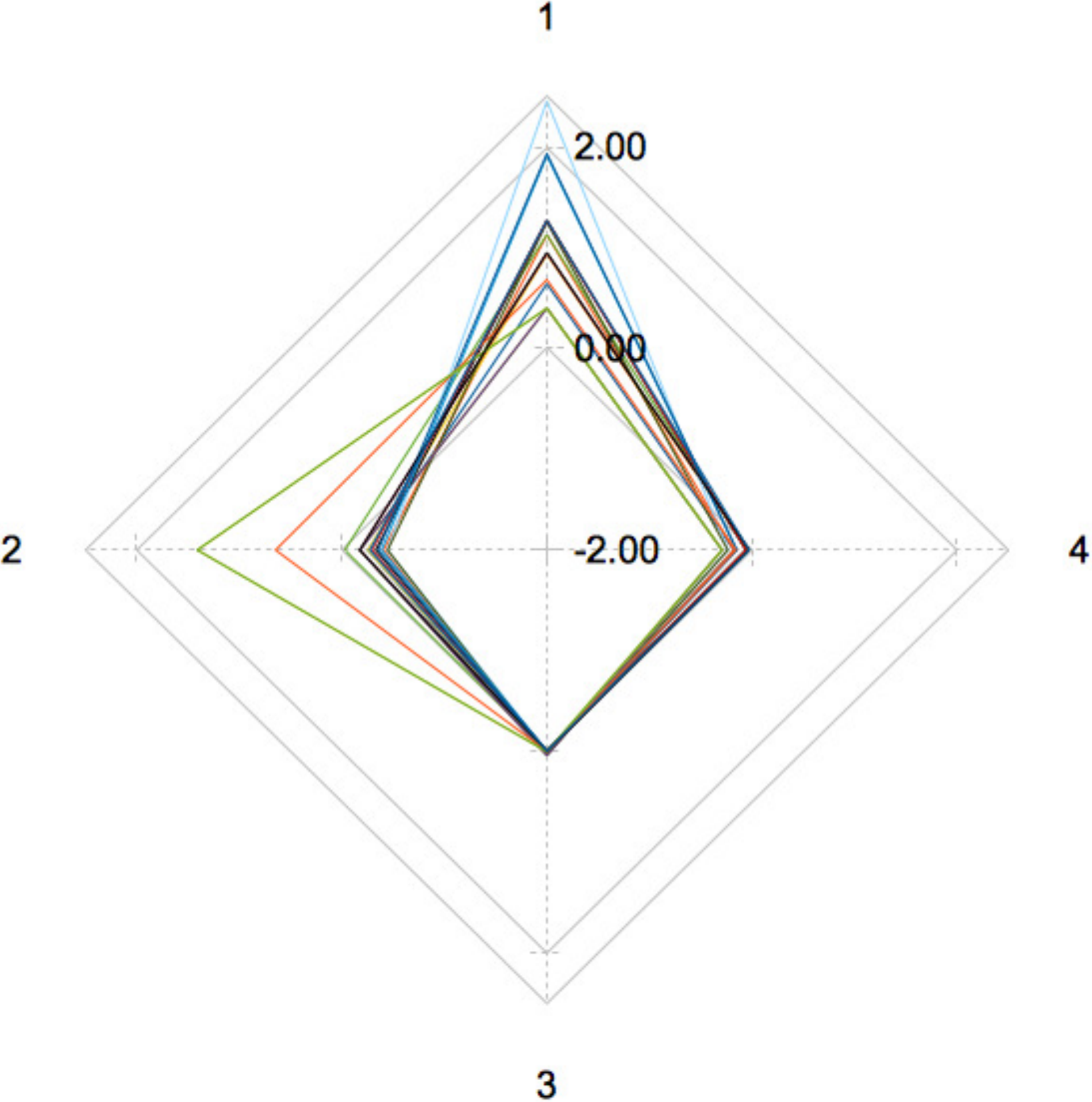}
\end{center}
\caption{Visualization of the 18 vertices of the upper image $\P\subset \R^4$ in Example \ref{ex36} by a radar chart.}
\label{fig_5}
\end{figure}

\begin{example}\label{ex36}
Consider $d=10$ assets with a given correlation structure and all other input parameters as in Example~\ref{AVAR2} above. Let $X$ be the payoff of an outperformance option with physical delivery as described in Example~\ref{AVAR1}. As the space of eligible assets we choose the space spanned by the first four assets, i.e. $M = \R^4 \times \{0\}^{6}$. The corresponding linear vector optimization problem has $4$ objectives and a matrix of size $5\,126 \times 51\,300$ with $197\,638$ nonzero entries. The ordering cone is $K_0^M$, which is strictly larger than $\R^4_+$ and generated by $12$ vectors. Then, $AV@R_\alpha(X)$ with $\alpha$ as in Example~\ref{AVAR2}, obtained as the upper image of linear vector optimization problem computed with the primal algorithm and $\varepsilon=10^{-2}$, has 18 vertices and 12 extreme directions. The vertices of $\P$ are visualized by a radar chart in Figure \ref{fig_5}.
\end{example}

The following numerical example refers to Example~\ref{ex_WCtheory} in the previous section.

\begin{example}\label{ex45}
Let us consider $d=9$ assets with a given correlation structure, all other input parameters as in Example~\ref{AVAR2} above (i.e. $m=3$), and the same basket option (basket call) with payoff $X$ as in Example~\ref{AVAR2}. We want to calculate $RWC(-X)$, the relaxed worst case risk measure at $-X$, as described in Example~\ref{ex_WCtheory} with parameter $\epsilon_i=500$ for $i\in\{3,\dots,9\}$ and zero otherwise. The cone $G$ can be seen as a worst case solvency cone and is chosen to be a conservative modification of $K_0$, where $\lambda$ is replaced by the larger transaction costs $\lambda_{wc}=\lambda+0.2$.

\begin{figure}[h]
\begin{center}
\includegraphics[scale=0.52]{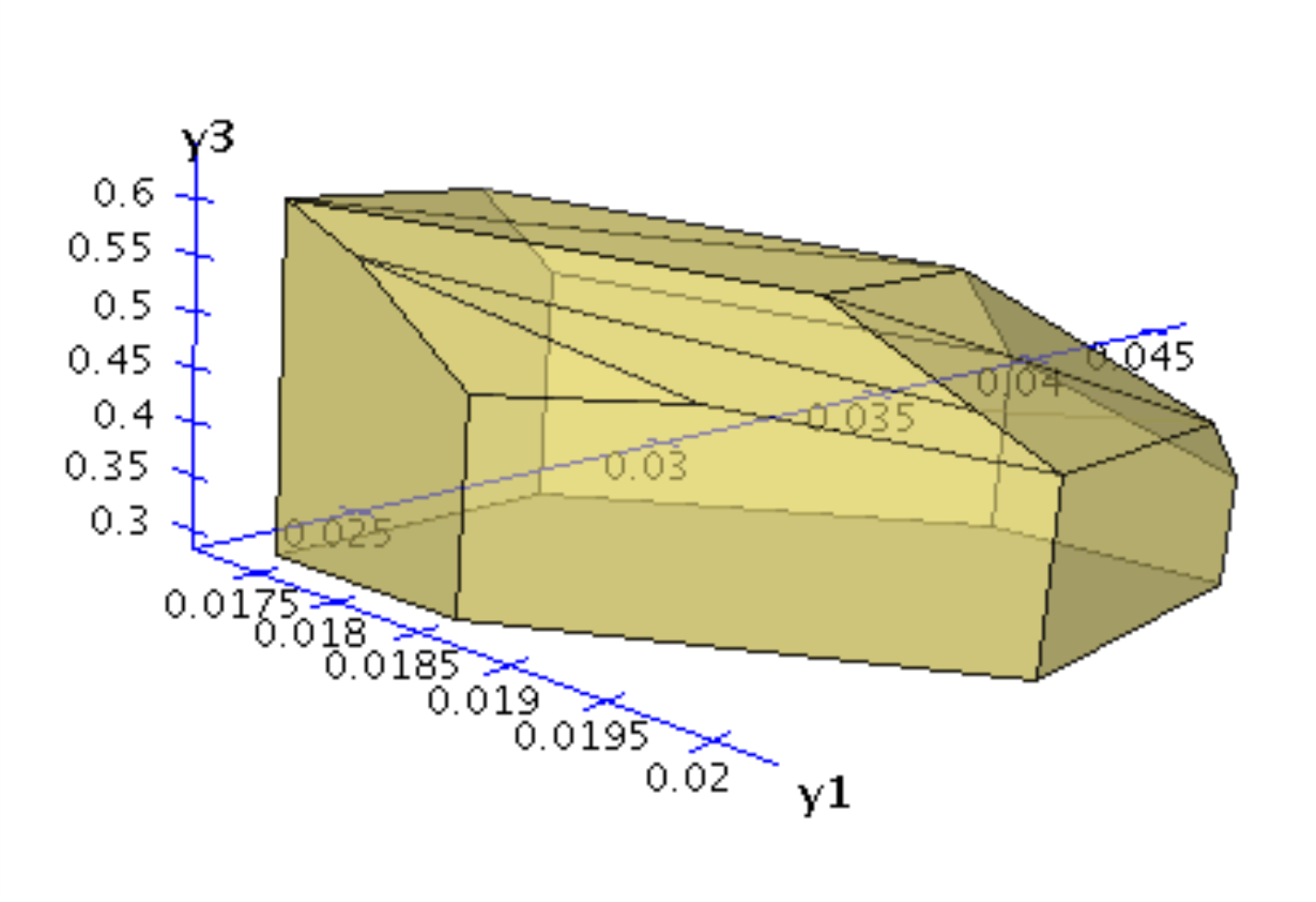}  
\includegraphics[scale=0.52]{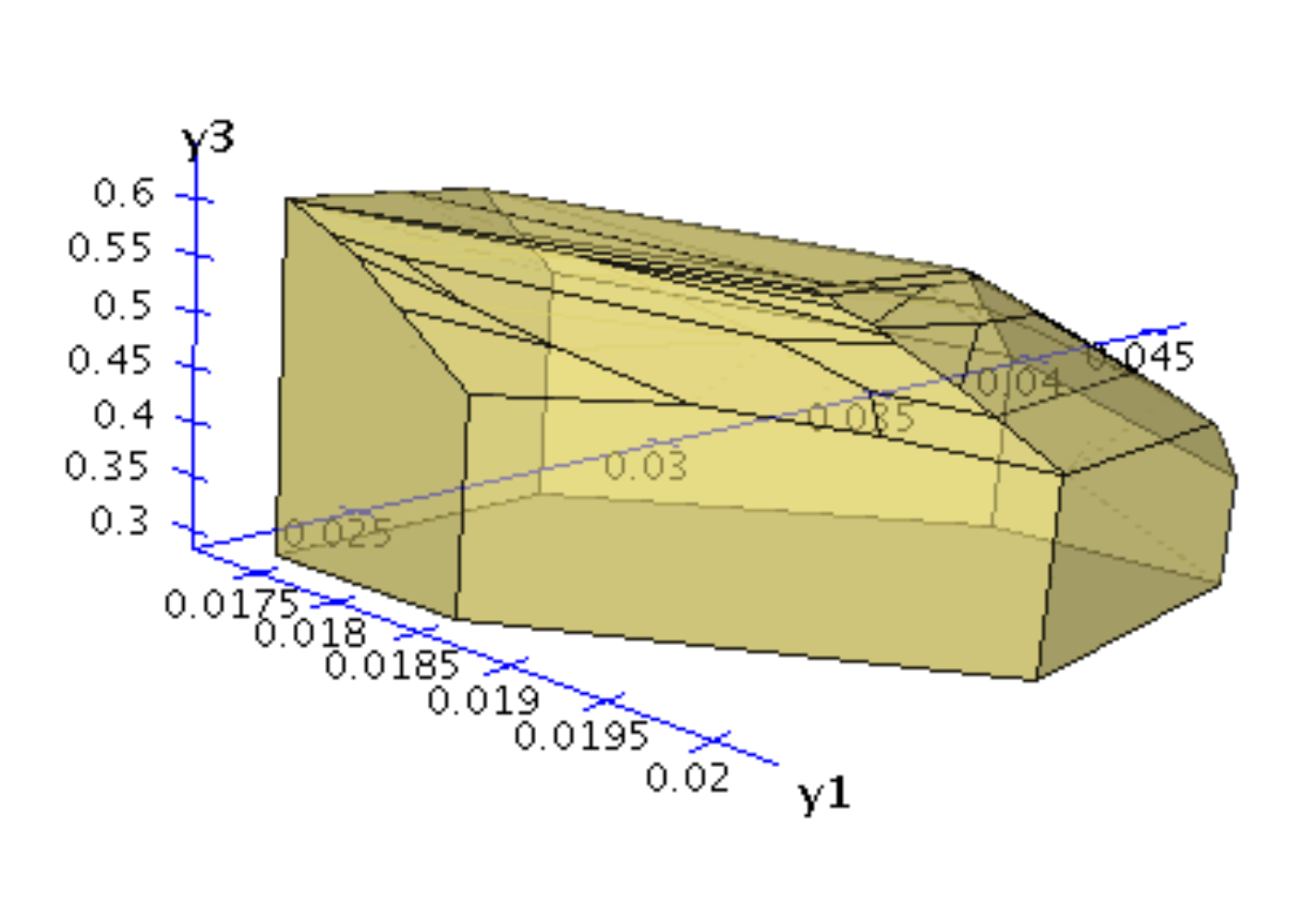}  
\end{center}
\caption{The lower image $\D^*$ in Example \ref{ex45} for $\varepsilon=10^{-2}$ (left) and $\varepsilon=10^{-3}$ (right) computed with the dual algorithm.}
\label{fig_6}
\end{figure}

$RWC(-X)$ corresponds to an upper good deal bound as it is a relaxed version of the set of superhedging portfolios. By considering certain small risks controlled by $\epsilon$ and $G$ as acceptable, the scalar superhedging price of $34.942464$ units of bond is reduced to $34.830995$ units of bond for $RWC(-X)$. The linear vector optimization problem to calculate $RWC(-X)$ has $3$ objectives and a matrix $B$ of size $4\,608 \times 36\,939$, which is sparse with $185\,856$ nonzero entries. 
The above prices in units of bond were obtained by solving linear vector optimization problems with both the primal and dual algorithm for $\varepsilon=10^{-4}$.

The following table shows some computational data and a comparison of the primal and dual algorithm. In Figure \ref{fig_6}, parts of the results are visualized.
\noindent
\begin{center} \begin{tabular}{|c|c|c|c|c|c|c|c|}
\hline $\varepsilon$ &  variant & total time& $|\bar S|$ & $|\bar T|$ & \# LPs & $t_{max}$ & $t_{max}/t_{aver}$ \\
\hline $10^{-2}$     & primal   & 113 secs  & 13       & 12        & 37     & 8.8 secs  & 3.8   \\ 
\hline $10^{-3}$     & primal   & 239 secs  & 68       & 37        & 123    & 9.0 secs  & 6.8   \\ 
\hline $10^{-4}$     & primal   & 506 secs  & 153      & 82        & 308    & 8.8 secs  & 8.6   \\ 
\hline $10^{-2}$     & dual     &  86 secs  & 7        & 20        & 37     & 5.8 secs  & 3.8   \\ 
\hline $10^{-3}$     & dual     & 193 secs  & 30       & 73        & 113    & 8.2 secs  & 4.8   \\ 
\hline $10^{-4}$     & dual     & 404 secs  & 74       & 136       & 256    & 5.8 secs  & 5.1   \\ 
\hline \end{tabular}
\end{center}
\end{example}

\noindent
\textbf{Acknowledgements.} We thank Dr Lizhen Shao for providing the data of Example \ref{ex07} taken from \cite{ShaEhr08} and we thank Professor Robert Vanderbei for supplying the data of Example \ref{ex06} taken from \cite{RusVan03}.

\bibliographystyle{abbrv}

\end{document}